\documentclass[12pt]{article}
\usepackage[]{amsmath,amssymb}
\usepackage{amscd}
\usepackage{latexsym}
\usepackage{cite}
\usepackage{amsthm}

\usepackage{amsfonts,  eucal} 
\usepackage[dvips]{pict2e}
\usepackage[dvipdfmx]{graphics}

\usepackage{cite}

\newtheorem{definition}{Definition}[section]
\newtheorem{theorem}[definition]{Theorem}
\newtheorem{lemma}[definition]{Lemma}
\newtheorem{corollary}[definition]{Corollary}
\newtheorem{remark}[definition]{Remark}
\newtheorem{example}[definition]{Example}
\newtheorem{conjecture}[definition]{Conjecture}
\newtheorem{problem}[definition]{Problem}
\newtheorem{note}[definition]{Note}
\newtheorem{assumption}[definition]{Assumption}
\newtheorem{proposition}[definition]{Proposition}

\begin{document}
\title{\bf 
Tridiagonal pairs, alternating elements,\\ 
and distance-regular graphs
}
\author{
Paul Terwilliger 
}
\date{}

\maketitle
\begin{abstract} The positive part $U^+_q$ of $U_q(\widehat{\mathfrak{sl}}_2)$ has a presentation with two generators $W_0$, $W_1$ and two relations called the $q$-Serre relations.
The algebra $U^+_q$ contains some elements, said to be alternating.
There are four kinds of alternating elements, denoted
$\lbrace  W_{-k}\rbrace_{k\in \mathbb N}$,
$\lbrace  W_{k+1}\rbrace_{k\in \mathbb N}$,
$\lbrace  G_{k+1}\rbrace_{k\in \mathbb N}$,
$\lbrace {\tilde G}_{k+1}\rbrace_{k \in \mathbb N}$. The alternating elements of each kind mutually commute.
A tridiagonal pair is an ordered pair of diagonalizable linear maps $A, A^*$ on a nonzero, finite-dimensional vector space $V$,
that each act in a (block) tridiagonal fashion on the eigenspaces of the other one. Let $A$, $A^*$ denote a tridiagonal pair on $V$.
Associated with this pair are
six well-known direct sum decompositions of $V$; these are the
eigenspace decompositions of $A$ and $A^*$, along  with four decompositions of $V$ that are often called split.
 In our main results, we assume that $A$, $A^*$ has $q$-Serre type. Under this assumption $A$, $A^*$
satisfy the $q$-Serre relations, and $V$ becomes an irreducible $U^+_q$-module on which $W_0=A$ and $W_1=A^*$.
We describe how the
alternating elements of $U^+_q$ act on the above six decompositions of $V$. We show that for each decomposition, every
alternating element acts in either a (block) diagonal, (block) upper bidiagonal, (block) lower bidiagonal, or (block) tridiagonal fashion.
We investigate  two special cases in detail. In the first case the eigenspaces of $A$ and $A^*$
all have dimension one. In the second case $A$ and $A^*$ are obtained by adjusting the adjacency matrix and a dual adjacency matrix
of a distance-regular graph that has classical parameters and is formally self-dual.

\bigskip


 \end{abstract}
\section{Introduction}
\noindent Our point of departure is the following remarkable fact. Let $\mathbb F$ denote a field, and fix a nonzero $b \in \mathbb F$
that is not a root of unity. Define an algebra over $\mathbb F$ by generators $W_0, W_1$ and relations
\begin{align} 
\lbrack W_0, \lbrack W_0, \lbrack W_0, W_1\rbrack_b \rbrack_{b^{-1}} \rbrack=0,
\qquad \qquad 
\lbrack W_1, \lbrack W_1, \lbrack W_1, W_0\rbrack_b \rbrack_{b^{-1}}
\rbrack=0,\label{eq:nqSerre}
\end{align}
where
\begin{align*}
\lbrack X, Y \rbrack=  XY- YX, \qquad \qquad
\lbrack X, Y \rbrack_b= b XY- YX.
\end{align*}
Using $W_0$, $W_1$ and the equations below, we recursively define some elements
\begin{align}
\label{eq:nWWGGnIntro}
\lbrace  W_{-k}\rbrace_{k\in \mathbb N}, \quad 
\lbrace  W_{k+1}\rbrace_{k\in \mathbb N}, \quad
\lbrace  G_{k+1}\rbrace_{k\in \mathbb N}, \quad
\lbrace {\tilde G}_{k+1}\rbrace_{k \in \mathbb N}
\end{align}
in the following order:
\begin{align*}
W_0, \quad W_1, \quad G_1, \quad \tilde G_1, \quad W_{-1}, \quad W_2, \quad
 G_2, \quad \tilde G_2, \quad W_{-2}, \quad W_3, \quad \ldots
\end{align*}
For $n\geq 1$,
\begin{align*}
&G_n = \frac{\sum_{k=0}^{n-1} W_{-k} W_{n-k} b^{-k} 
-
\sum_{k=1}^{n-1} G_k  \tilde G_{n-k}  b^{-k} }{1+b^{-n}}
+ 
\frac{\lbrack W_n, W_0\rbrack }{(1+b^{-n})(1-b^{-1})},
\\
\tilde G_n &= G_n + \frac{ \lbrack W_0,W_n\rbrack}{1-b^{-1}},
\qquad \quad 
W_{-n} = \frac{ \lbrack W_0, G_n\rbrack_b}{b-1}, 
\qquad \quad 
W_{n+1} = \frac{\lbrack G_n, W_1\rbrack_b}{b-1}. 
\end{align*}
The remarkable fact is that for $k, \ell \in \mathbb N$,
\begin{align*}
&
\lbrack  W_{-k},  W_{-\ell}\rbrack=0,  \qquad 
\lbrack  W_{k+1},  W_{\ell+1}\rbrack= 0,
\\
&
\lbrack  G_{k+1},  G_{\ell+1}\rbrack=0,
\qquad 
\lbrack {\tilde G}_{k+1},  {\tilde G}_{\ell+1}\rbrack= 0.
\end{align*} 
The above fact is from
\cite[Proposition~5.10]{alternating}.
The relations \eqref{eq:nqSerre} are called the $q$-Serre relations, where $q^2=b$. The defined algebra is denoted by $U^+_q$,
and called the positive part of  $U_q(\widehat{\mathfrak{sl}}_2)$; see Definition \ref{def:nposp} below.
 The elements \eqref{eq:nWWGGnIntro} are called the alternating elements of $U^+_q$ \cite[Definition~5.1]{alternating}.
These elements satisfy the relations in
\cite[Proposition~5.7]{alternating},
\cite[Proposition~5.10]{alternating},
\cite[Proposition~5.11]{alternating},
\cite[Proposition~6.3]{alternating},
\cite[Proposition~8.1]{alternating}.
\medskip

\noindent 
The alternating elements get their name in the following way. Let $x$, $y$ denote noncommuting indeterminates. 
Let $F$ denote the free algebra  with generators $x,y$.  By a letter we mean $x$ or $y$. For $n\in \mathbb N$,
by a word of length $n$  we mean a product of letters $a_1a_2\cdots a_n$. The words form a linear basis for $F$.
In
\cite{rosso1, rosso} M. Rosso introduced an algebra structure on $F$ called the $q$-shuffle algebra.
For letters $u,v$ their $q$-shuffle product is given by
$u\star v = uv+q^{(u,v) }vu$, where
$(u,v)=2$ if $u=v$ and $(u,v)=-2$ if
 $u\not=v$. In \cite[Theorem~15]{rosso},
Rosso gave an injective algebra homomorphism 
from $U^+_q$ into the $q$-shuffle algebra
$F$, that sends $W_0\mapsto x$ and $W_1\mapsto y$.
By
\cite[Definition~5.2]{alternating}
the homomorphism sends
\begin{align*}
&W_0 \mapsto x, \qquad W_{-1} \mapsto xyx, \qquad W_{-2} \mapsto xyxyx, \qquad \ldots
\\
&W_1 \mapsto y, \qquad W_{2} \mapsto yxy, \qquad W_{3} \mapsto yxyxy, \qquad \ldots
\\
&G_{1} \mapsto yx, \qquad G_{2} \mapsto yxyx,  \qquad G_3 \mapsto yxyxyx, \qquad \ldots
\\
&\tilde G_{1} \mapsto xy, \qquad \tilde G_{2} \mapsto 
xyxy,\qquad \tilde G_3 \mapsto xyxyxy, \qquad \ldots
\end{align*}
We used this homomorphism
to obtain the relations in \cite{alternating} mentioned above.
\medskip

\noindent Our discovery  of the alternating elements \cite{alternating} was motivated by the groundbreaking work of Baseilhac/Koizumi \cite{BK05}
and Baseilhac/Shigechi \cite{basnc}
concerning the $q$-Onsager algebra $O_q$ \cite{bas1,qSerre}.
The algebras $U^+_q$, $O_q$ have a similar structure; they both belong to a family of
algebras called the tridiagonal algebras  \cite[Definition~3.9]{qSerre}. The algebra $U^+_q$ is considerably less complicated than $O_q$, and
it is natural to view $U^+_q$ as a toy model for $O_q$. With this view in mind, we now discuss \cite{BK05}, \cite{basnc}.
In \cite{BK05} Baseilhac and Koizumi investigate  boundary
integrable systems with hidden symmetries. In \cite[Section~2.1]{BK05} they use an RKRK reflection
equation to define an algebra \cite[Line (4)]{BK05} that is now denoted by $\mathcal O_q$ and called the alternating central extension of $ O_q$.
In \cite[Definition~3.1]{basnc}  Baseilhac and Shigechi give a presentation of  $\mathcal O_q$ by generators and relations.
 This presentation resembles
 \cite[Proposition~5.7]{alternating},
\cite[Proposition~5.10]{alternating}. Moreover \cite[Proposition~3.1]{basnc} resembles
\cite[Proposition~5.11]{alternating}.
After \cite{BK05} and \cite{basnc} appeared, we tried to understand these works by searching for analogous results about $U^+_q$. In this
search 
we considered the Rosso embedding of $U^+_q$ into a $q$-shuffle algebra, and this lead us to discover the alternating elements of $U^+_q$. We believe that this
discovery would not have occurred without the guides \cite{BK05}, \cite{basnc}.
\medskip

\noindent Next we summarize what is known about the alternating elements of $U^+_q$. 
We will refer to a Poincar\'e-Birkhoff-Witt (or PBW) basis for $U^+_q$ due to Damiani \cite{damiani}, and a related PBW basis for $U^+_q$ due to Beck \cite[Proposition~6.1]{beck}.
In \cite[Theorem~10.1]{alternating} we showed that the alternating elements 
$\lbrace  W_{-k}\rbrace_{k\in \mathbb N}$, 
$\lbrace {\tilde G}_{k+1}\rbrace_{k \in \mathbb N}$,
$\lbrace  W_{k+1}\rbrace_{k\in \mathbb N}$ form a PBW basis for $U^+_q$; this PBW basis is said to be alternating.
In \cite[Section 11]{alternating} we related the alternating PBW basis to the Damiani PBW basis.
In \cite[Proposition~9.11]{beckPBW} we related the alternating PBW basis  to the Beck PBW basis.
In \cite[Theorem~9.15]{alternating} we expressed the elements $\lbrace G_{k+1} \rbrace_{k \in \mathbb N}$ in terms of the alternating PBW basis.
This expression involves some elements $\lbrace D_n \rbrace_{n \in \mathbb N}$ that we defined recursively; in \cite[Theorem~1.11]{peterRuan}
Chenwei Ruan displayed the elements $\lbrace D_n \rbrace_{n \in \mathbb N}$  in closed form.
In \cite[Theorem~2.10]{basSecond}, Baseilhac interprets the alternating elements of $U^+_q$ using an RKRK reflection equation.
In \cite[Section~3]{basSecond} Baseilhac relates the alternating elements of $U^+_q$  to the Drinfeld generators of $U_q(\widehat{\mathfrak{sl}}_2)$.
In \cite{altCE} we used  the alternating elements of $U^+_q$ 
to construct an algebra $\mathcal U^+_q$, called the alternating central extension of $U^+_q$. The algebra $\mathcal U^+_q$ is discussed in \cite{basSecond, compactUqp, factorUq}.
\medskip

\noindent In the present paper, we interpret the alternating elements of $U^+_q$ using the concept of a tridiagonal pair (or TD pair); see Definition \ref{def:tdp} below.
Roughly speaking, a TD pair is an ordered pair $A$, $A^*$ of 
diagonalizable linear maps on a nonzero finite-dimensional vector space $V$, that each act on the eigenspaces of the
other one in a (block) tridiagonal fashion. Let $A$, $A^*$ denote a TD pair on $V$. Associated with this TD pair are
six well-known direct sum decompositions of $V$; these are the
eigenspace decompositions of $A$ and $A^*$, along  with four decompositions of $V$ that are often called split \cite[Section~4]{TD00}. On each split decomposition of $V$,
one of $A$,  $A^*$ acts  in a (block) upper bidiagonal fashion, and the other one acts in a  (block) lower bidiagonal fashion \cite[Lemma~5.1]{tdanduqsl2hat}.
As we will see, it is natural to  represent the six decompositions of $V$ by the edges of a tetrahedron, with the eigenspace decompositions of 
$A$ and $A^*$ represented by a pair of opposite edges. The resulting diagram is called the tetrahedron diagram.
In our main results, we assume that the TD pair $A$, $A^*$ has $q$-Serre type \cite[Definition~2.6]{nonnil}. Under this assumption
 $A$, $A^*$ satisfy the $q$-Serre relations, so  $V$ becomes 
a $U^+_q$-module on which $W_0=A$  and $W_1=A^*$. It turns out that the $U^+_q$-module $V$ is irreducible; see Corollary \ref{cor:modU} below.
We describe how the
alternating elements of $U^+_q$ act on the six decompositions of $V$ from the tetrahedron diagram. We show that for each decomposition, every
alternating element acts in either a (block) diagonal, (block) upper bidiagonal, (block) lower bidiagonal, or (block) tridiagonal fashion.
Our main results on this topic are Theorems \ref{thm:WWaction}, \ref{thm:GGaction}. 
In view of the remarkable fact that we began with, it is natural to seek bases for $V$ of the following types:
\begin{enumerate}
\item[\rm (i)] a basis of common eigenvectors for $\lbrace W_{-k} \rbrace_{k\in \mathbb N}$;
\item[\rm (ii)] a basis of common eigenvectors for $\lbrace W_{k+1} \rbrace_{k\in \mathbb N}$;
\item[\rm (iii)] a basis of common eigenvectors for $\lbrace G_{k+1} \rbrace_{k\in \mathbb N}$;
\item[\rm (iv)] a basis of common eigenvectors for $\lbrace {\tilde G}_{k+1} \rbrace_{k\in \mathbb N}$.
\end{enumerate}
The above bases exist sometimes, but not always. We investigate  two cases in which the bases exist.
In the first  case, we assume that  $A$, $A^*$ is a Leonard pair; this means
that the eigenspaces of $A$ and $A^*$ all have dimension one. We show that for $k \in \mathbb N$
the element $W_{-k} $ (resp. $W_{k+1}$) acts on $V$ as a linear combination of $A, I$ (resp. $A^*, I$). We obtain a similar result for
$G_{k+1}$ and $\tilde G_{k+1}$. For the linear combinations that show up, we express the coefficients using a generating function. Our main results on this topic are
Theorems \ref{thm:third}, \ref{thm:exp}. 
For the second case, start with
a distance-regular graph $\Gamma$ that has diameter $d\geq 3$ and classical parameters $(d,b,\alpha, \beta)$ with $b \not=1$ and $\alpha = b-1$;
the condition on $\alpha$ implies that $\Gamma$ is formally self-dual in the sense of \cite[p.~49]{BCN}.
After an affine adjustment, the adjacency matrix and any dual adjacency matrix become a pair of matrices $A$ and $A^*$ that satisfy the $q$-Serre relations.
The corresponding subconstituent algebra $\mathbb T$ is generated by $A, A^*$. It is known that $A, A^*$ act on each irreducible $\mathbb T$-module as a TD
pair of $q$-Serre type \cite[Lemma~9.4]{drg}.  We show that each irreducible $\mathbb T$-module has  bases of type (i)--(iv) above. Our main
results on this topic are Theorems \ref{thm:orthog},  \ref{thm:orthog2}. We finish the paper with some open problems.
\medskip

\noindent 
This paper is organized as follows.
Section 2 contains some preliminaries.
In Section 3 we define  a TD pair and TD system, and provide some basic facts about these objects.
In Section 4 we discuss a TD system from the point of view of flags and decompositions.
In Section 5 we define the tetrahedron diagram and use it to describe a TD system.
In Section 6 we give some slightly technical comments about flags and decompositions.
In Section 7 we describe the tridiagonal relations.
In Section 8 we review the algebra $U^+_q$ and its alternating elements.
 Section 9 contains our main results about TD systems of $q$-Serre type.
Section 10 contains our main results about Leonard systems of $q$-Serre type.
Section 11 contains our main results about distance-regular graphs.
In Section 12 we give some  open problems.

\section{Preliminaries}
We now begin our formal argument.
The following concepts and notation will be used throughout the paper. Recall the natural numbers $\mathbb N= \lbrace 0,1,2,\ldots\rbrace$.
We will be discussing finite sequences. For $d \in \mathbb N$ a sequence $x_0, x_1, \ldots, x_d$ will be denoted by $\lbrace x_i \rbrace_{i=0}^d$. By the {\it inversion} of $\lbrace x_i\rbrace_{i=0}^d$ we mean
the sequence $\lbrace x_{d-i} \rbrace_{i=0}^d$.
Let $\mathbb F$ denote a field. 
Every vector space discussed in this paper is understood to be over $\mathbb F$. Every algebra discussed in this paper is 
understood to be associative, over $\mathbb F$, and have a multiplicative identity.
A subalgebra has the same multiplicative identity as the parent algebra.
Throughout the paper, $V$ denotes a nonzero vector space with finite dimension. The algebra ${\rm End}(V)$ 
consists of the $\mathbb F$-linear maps from $V$ to $V$.
An element $A \in {\rm End}(V)$ is said to be {\it diagonalizable} whenever
$V$ is spanned by the eigenspaces of $A$. Assume that $A$ is diagonalizable, and let $\lbrace V_i \rbrace_{i=0}^d$ 
denote an ordering of the eigenspaces of $A$. For $0 \leq i \leq d$ let $\theta_i$ denote the eigenvalue of $A$ for $V_i$.
For $0 \leq i \leq d$ define $E_i \in {\rm End}(V)$ such that
$(E_i-I) V_i=0$ and $E_iV_j=0$ if $j \not=i$ $(0 \leq j \leq d)$. Thus $E_i$ is the projection from $V$ onto $V_i$. We call $E_i$ the {\it primitive idempotent} of $A$ associated with $V_i$ (or $\theta_i$).
By linear algebra 
(i) $A = \sum_{i=0}^d \theta_i E_i$; 
(ii) $E_i E_j = \delta_{i,j} E_i$ $ (0 \leq i,j\leq d)$;
(iii) $I = \sum_{i=0}^d E_i$; 
(iv) $V_i = E_iV$ $ (0 \leq i \leq d)$; 
(v) $AE_i = \theta_i E_i = E_iA$ $(0 \leq i \leq d)$. Moreover
\begin{align*}
  E_i=\prod_{\stackrel{0 \leq j \leq d}{j \neq i}}
          \frac{A-\theta_jI}{\theta_i-\theta_j} \qquad \qquad (0 \leq i \leq d).
\end{align*}
Let $\langle A \rangle $ denote the subalgebra of ${\rm End}(V)$ generated by $A$. 
The elements $\lbrace A^i \rbrace_{i=0}^d $ form a basis for $\langle A \rangle $, and
$0=\prod_{i=0}^d (A-\theta_i I)$. Moreover the elements $\lbrace E_i \rbrace_{i=0}^d$ form a basis for $\langle A \rangle$.

\section{Tridiagonal pairs and tridiagonal systems}
\noindent 
In this section we recall the notions of a tridiagonal pair and a tridiagonal system. We give some basic facts about these objects.

\begin{definition}  {\rm (See
\cite[Definition 1.1]{TD00}.)}
\label{def:tdp}
\rm
A  {\it tridiagonal pair} (or {\it TD pair}) on $V$
is an ordered pair $A, A^*$ of elements in ${\rm End}(V)$ that satisfy
the following four conditions.
\begin{enumerate}
\item[\rm (i)]  Each of $A,A^*$ is diagonalizable.
\item[\rm (ii)] There exists an ordering $\lbrace V_i \rbrace_{i=0}^d $ of the  
eigenspaces of $A$ such that 
\begin{equation}
A^* V_i \subseteq V_{i-1} + V_i+ V_{i+1} \qquad \qquad (0 \leq i \leq d),
\label{eq:t1}
\end{equation}
where $V_{-1} = 0$ and $V_{d+1}= 0$.
\item[\rm (iii)]There exists an ordering $\lbrace V^*_i \rbrace_{i=0}^\delta$ of
the  
eigenspaces of $A^*$ such that 
\begin{equation}
A V^*_i \subseteq V^*_{i-1} + V^*_i+ V^*_{i+1} \qquad \qquad (0 \leq i \leq \delta),
\label{eq:t2}
\end{equation}
where $V^*_{-1} = 0$ and $V^*_{\delta+1}= 0$.
\item[\rm (iv)]  There does not exist a subspace $W$ of $V$ such  that $AW\subseteq W$,
$A^*W\subseteq W$, $W\not=0$, $W\not=V$.
\end{enumerate}
\end{definition}

\begin{note} \rm
According to a common 
notational convention, $A^*$ denotes the conjugate transpose
of $A$.
 We are not using this convention.
In a TD pair $A,A^*$ the maps
$A$ and $A^*$ are arbitrary subject to (i)--(iv) above.
\end{note}



\noindent We refer the reader to \cite{class}, \cite{TD00}, \cite{augIto}, \cite{totBip} for background and historical remarks about TD pairs.

\begin{lemma} \label{lem:affine}
For a TD pair $A, A^*$  on $V$ and  scalars $r,r^*,s,s^*\in \mathbb F$ with $rr^*\not=0$, the pair $rA+sI$, $r^*A^*+s^*I$ 
is a TD pair on $V$.
\end{lemma}
\begin{proof} Routine.
\end{proof}

\noindent We have been discussing TD pairs. There is a related notion, called a TD system. To define a TD system, we will use the following concept. Let $A, A^*$ denote a TD pair on $V$.
An ordering of the eigenspaces of $A$ (resp. $A^*$)
is called
{\it standard} whenever it satisfies
(\ref{eq:t1}) (resp. 
(\ref{eq:t2})). 
We comment on the uniqueness of the standard ordering.
Let 
 $\lbrace V_i \rbrace_{i=0}^d$ denote a standard ordering
of the eigenspaces of $A$.
Then by \cite[Lemma~2.4]{TD00}, the inverted ordering 
 $\lbrace V_{d-i} \rbrace_{i=0}^d$ 
is also standard and no further ordering is standard.
A similar result holds for the 
 eigenspaces of $A^*$. An ordering of the primitive idempotents of $A$ (resp. $A^*$) is said to be {\it standard} whenever
 the corresponding ordering of the eigenspaces of $A$ (resp. $A^*$) is standard.
\begin{definition} \rm (See  \cite[Definition~2.1]{TD00}, \cite[Definition~2.1]{nomsharp}.)
 \label{def:TDsystem} 
A  {\it tridiagonal system} (or {\it  $TD$ system}) on $V$ is a sequence
\begin{align*}
 \Phi=(A;\lbrace E_i\rbrace_{i=0}^d;A^*;\lbrace E^*_i\rbrace_{i=0}^\delta)
\end{align*}
that satisfies the following three conditions:
\begin{enumerate}
\item[\rm (i)]
$A,A^*$ is a TD pair on $V$;
\item[\rm (ii)]
$\lbrace E_i\rbrace_{i=0}^d$ is a standard ordering
of the primitive idempotents of $A$;
\item[\rm (iii)]
$\lbrace E^*_i\rbrace _{i=0}^\delta$ is a standard ordering
of the primitive idempotents of $A^*$.
\end{enumerate}
\end{definition}
\noindent We have a comment.
\begin{lemma} {\rm (See \cite[Section~3]{TD00}.)} Consider a TD system on $V$:
\begin{align*}
\Phi = (
A; \lbrace E_i \rbrace_{i=0}^d;
A^*; \lbrace E^*_i \rbrace_{i=0}^\delta).
\end{align*}
For scalars $r,r^*,s,s^*\in \mathbb F$ such that $rr^*\not=0$, 
 the sequence
\begin{align*}
 (
r A+sI; \lbrace E_i \rbrace_{i=0}^d;
r^*A^*+s^* I; \lbrace E^*_i \rbrace_{i=0}^\delta)
\end{align*}
is a TD system on $V$.
Moreover, each of the following is a TD system on $V$:
\begin{align*}
\Phi^* &= (
A^*; \lbrace E^*_i \rbrace_{i=0}^{\delta};
A; \lbrace E_i \rbrace_{i=0}^d),
\\ 
\Phi^\downarrow &= 
(A; \lbrace E_i \rbrace_{i=0}^d;
A^*; \lbrace E^*_{\delta-i} \rbrace_{i=0}^\delta),
\\
\Phi^\Downarrow &= (
A; \lbrace E_{d-i} \rbrace_{i=0}^d;
A^*; \lbrace E^*_i \rbrace_{i=0}^\delta).
\end{align*}
\end{lemma}
\medskip

\noindent From now until the end of Section 10, we fix a TD system on $V$:
\begin{align} \label{eq:TDS}
 \Phi=(A;\lbrace E_i\rbrace_{i=0}^d;A^*;\lbrace E^*_i\rbrace_{i=0}^\delta).
\end{align}

\begin{lemma} {\rm (See \cite[Lemma~2.4]{TD00}, \cite[Lemma~2.5]{nomTowards}.)} 
\label{lem:tpr} Referring to the TD system $\Phi$, we have
\begin{align*}
&
E_i A^* E_j =
\begin{cases}
0 & {\mbox{\rm if $|i-j|>1$}};\\
\not=0 &  {\mbox{\rm if $|i-j|=1$
}}
\end{cases}
\qquad \qquad (0 \leq i,j \leq d),
\\
&
E^*_i A E^*_j =
\begin{cases}
0 & {\mbox{\rm if $|i-j|> 1$}};\\
\not=0 &  {\mbox{\rm if $|i-j|=1$
}}
\end{cases}
\qquad \qquad (0 \leq i,j \leq \delta).
\end{align*}
\end{lemma}
\bigskip


\begin{definition}\rm {(See \cite[Definition~3.1]{TD00}.)}
Referring to  the TD system $\Phi$, for $0 \leq i\leq d$ let $\theta_i$ 
denote the eigenvalue of $A$ associated
with $E_i$. 
 We call $\lbrace \theta_i \rbrace_{i=0}^d$
the {\it eigenvalue sequence} of $\Phi$.
For $0 \leq i\leq \delta$ let $\theta^*_i$ 
denote the eigenvalue of $A^*$ associated
with $E^*_i$. 
 We call $\lbrace \theta^*_i \rbrace_{i=0}^\delta$
the {\it dual eigenvalue sequence} of $\Phi$. 
\end{definition}
\noindent We emphasize that
 $\lbrace \theta_i \rbrace_{i=0}^d$ are mutually distinct, and  $\lbrace \theta^*_i \rbrace_{i=0}^\delta$ are mutually distinct.
 \medskip
 
\noindent Referring to the TD system $\Phi$, it is shown in \cite[Lemma~4.5]{TD00} that $d=\delta$. This result is nontrivial, and can be obtained using some ideas that will play a role
in our main results. We will give a short proof in order to introduce these ideas.
\medskip

\noindent 
For $0\leq i\leq \delta$ and $0\leq j \leq d$, define
\begin{align}
V_{i,j} = (E^*_0V+E^*_1V+ \cdots + E^*_iV) \cap (E_0V+E_1V+\cdots + E_{j}V).
\label{eq:Vij}
\end{align}

\noindent For example, $V_{0,d} = 
 E^*_0V$ and
$V_{\delta,0} =  E_0V$.
\noindent Define
\begin{align}
N = {\rm min} \lbrace i+j \vert  V_{i,j}\not=0\rbrace.
\label{eq:N}
\end{align}
\noindent 
We have  $N\leq d$ because
$V_{0,d}\not=0$. We have
$N\leq \delta$ because $V_{\delta, 0}\not=0$.
  For $0 \leq i \leq N$ abbreviate $U_i = V_{i,N-i}$, so that
 \begin{align}\label{eq:abbrev}
      U_i = (E^*_0V+E^*_1V+ \cdots + E^*_iV) \cap (E_0V+E_1V+\cdots + E_{N-i}V).
      \end{align}

\noindent In a moment we will show that $N=d=\delta$.
\begin{lemma} \label{lem:RL}  {\rm (See \cite[Lemma~4.4]{TD00}.)} We have
\begin{enumerate}
\item[\rm (i)] 
$(A^*-\theta^*_{i} I)U_i \subseteq U_{i-1} $ for $1 \leq i \leq N$;
\item[\rm (ii)]
$(A^*-\theta^*_0 I)U_0 =0$;
\item[\rm (iii)] $(A-\theta_{N-i}I)U_i \subseteq U_{i+1}$ for $0 \leq i \leq N-1$;
\item[\rm(iv)]
$(A-\theta_0 I)U_N =0$.
\end{enumerate}
\end{lemma}
\begin{proof} (i) Use \eqref{eq:abbrev} along with the following observations.
\begin{align*}
 (A^*- \theta^*_i I)(E^*_0V+ \cdots + E^*_i V) = E^*_0V+ \cdots + E^*_{i-1}V.
 \end{align*}
 By \eqref{eq:t1},
 \begin{align*}
 (A^*-\theta^*_i I) (E_0V+\cdots + E_{N-i}V) \subseteq E_0V+\cdots + E_{N-i+1}V.
 \end{align*}
 \\
 \noindent (ii) Setting $i=0$ in \eqref{eq:abbrev} we  have $U_0 \subseteq E^*_0V$.
 \\
 \noindent (iii), (iv) Apply (i), (ii) to $\Phi^*$.
\end{proof}

\begin{corollary} \label{lem:RL2} We have
\begin{enumerate}
\item[\rm (i)] 
$A^*U_i \subseteq U_{i-1}+ U_i $ for $1 \leq i \leq N$;
\item[\rm (ii)]
$A^*U_0 \subseteq U_0$;
\item[\rm (iii)] $A U_i \subseteq U_i + U_{i+1}$ for $0 \leq i \leq N-1$;
\item[\rm(iv)]
$ A U_N \subseteq U_N$.
\end{enumerate}
\end{corollary}
\begin{proof} By Lemma \ref{lem:RL}.
\end{proof}

\begin{lemma} \label{lem:Vsum} 
 $V=\sum_{i=0}^N U_i$.
\end{lemma}
\begin{proof} Define $U=\sum_{i=0}^N U_i$. We have $U \not=0$, since at least one of $\lbrace U_i \rbrace_{i=0}^N$ is nonzero by  \eqref{eq:N}.
We have $A U \subseteq U$ and $A^*U \subseteq U$ by Corollary \ref{lem:RL2}. Therefore $U=V$ in view of Definition \ref{def:tdp}(iv).
\end{proof}

 \begin{lemma} {\rm (See \cite[Lemma~4.5]{TD00}.)} $N=d=\delta$.
 \end{lemma} 
 \begin{proof}   We mentioned below \eqref{eq:N} that $N\leq d$ and $N \leq \delta$.
 By \eqref{eq:abbrev} we have $U_i \subseteq E_0V+\cdots + E_{N-i}V$ for $0 \leq i \leq N$. Consequently $\sum_{i=0}^N U_i \subseteq \sum_{i=0}^N E_iV$.
 Observe that
 \begin{align*}
\sum_{i=0}^d E_iV = V =\sum_{i=0}^N U_i \subseteq \sum_{i=0}^N E_iV.
 \end{align*}
 Therefore $N=d$. 
 By \eqref{eq:abbrev} we have $U_i \subseteq E^*_0V+\cdots + E^*_iV$ for $0 \leq i \leq N$. Consequently $\sum_{i=0}^N U_i \subseteq \sum_{i=0}^N E^*_iV$. Observe that
 \begin{align*}
  \sum_{i=0}^\delta E^*_iV = V = \sum_{i=0}^N U_i \subseteq \sum_{i=0}^N E^*_iV.
  \end{align*}
 Therefore $N=\delta$.
 \end{proof}
 
\begin{definition}\label{def:diam} \rm We just showed $d=\delta$; this common value is called the {\it diameter of $\Phi$}.
We say that $\Phi$ is {\it trivial} whenever it has diameter 0.
\end{definition}

 \noindent Note that $\Phi$ is trivial if and only if ${\rm dim}\,V=1$. For the rest of this paper, we assume that $\Phi$ is nontrivial.
 \medskip
 
\noindent In the next few results, we describe the subspaces $\lbrace U_i \rbrace_{i=0}^d$ in more  detail.

\begin{lemma} \label{lem:dVsum} {\rm (See \cite[Theorem~4.6]{TD00}.)} 
The sum $V=\sum_{i=0}^d U_i$ is direct.
\end{lemma}
\begin{proof} It suffices to show that $(U_0 + \cdots + U_{i-1} ) \cap U_i = 0$ for $1 \leq i \leq d$. Let $i$ be given.
By \eqref{eq:abbrev},
\begin{align*}
U_0 + \cdots + U_{i-1} \subseteq E^*_0V+\cdots + E^*_{i-1}V, \qquad \qquad  U_i \subseteq E_0V + \cdots + E_{d-i}V.
\end{align*}
 By \eqref{eq:Vij},
\begin{align*}
(E^*_0V+\cdots + E^*_{i-1}V ) \cap (E_0V+ \cdots + E_{d-i}V) = V_{i-1,d-i}.
\end{align*}
Note that $V_{i-1,d-i}=0$ by \eqref{eq:N} and since $i-1+d-i=d-1<d=N$.
By these comments $(U_0 + \cdots + U_{i-1} ) \cap U_i = 0$.
\end{proof}

\begin{lemma} \label{lem:split} {\rm (See \cite[Theorem~4.6]{TD00}.)} 
For $0 \leq i \leq d$,
\begin{enumerate}
\item[\rm (i)] $U_0 + \cdots + U_i = E^*_0V+\cdots + E^*_iV$;
\item[\rm (ii)] $U_{i} + \cdots + U_d = E_0V+\cdots + E_{d-i}V$.
\end{enumerate}
\end{lemma}
\begin{proof} (i) The inclusion $\subseteq $ follows from \eqref{eq:abbrev}. Next we obtain the inclusion $\supseteq$. Define $X_i = \prod_{\ell=i+1}^d (A^* - \theta^*_\ell I)$, and note that
 $X_i V = E^*_0V+\cdots + E^*_iV$. Using Lemma  \ref{lem:RL}(i),(ii) we have $U_0 + \cdots + U_i \supseteq X_iU_j$ for $0 \leq j \leq d$. Therefore $U_0 + \cdots + U_i \supseteq X_iV$.
By these comments $U_0 + \cdots + U_i \supseteq E^*_0V+\cdots + E^*_iV$. We have obtained the inclusion $\supseteq$.
\\
\noindent (ii) Apply (i) to $\Phi^*$.
\end{proof}

\noindent Next we define some maps that will help us describe how the $\lbrace U_i\rbrace_{i=0}^d$ are related to the
$\lbrace E_iV\rbrace_{i=0}^d$ and $\lbrace E^*_iV\rbrace_{i=0}^d$.
For $0 \leq i \leq d$ define $E^\vee_i \in {\rm End}(V)$ such that
$(E^\vee_i - I ) U_i =0$ and $E^\vee_i U_j = 0$ if $i \not=j$ $(0 \leq j \leq d)$.
By linear algebra, 
$E^\vee_i E^\vee_j = \delta_{i,j} E^\vee_i$ for $0 \leq i,j\leq d$, and $I = \sum_{i=0}^d E^\vee_i$. By Lemma \ref{lem:split} or \cite[Lemma~5.4]{TD00}
we find that for $0 \leq i < j \leq d$,
\begin{align*}
 E^*_j E^\vee_i=0,
\qquad E^\vee_j E^*_i=0, \qquad 
E_{d-i} E^\vee_j=0, \qquad E^\vee_{i} E_{d-j}=0.
\end{align*}
Consequently  for $0 \leq i  \leq d$,
\begin{align}
& E^\vee_i E^*_i E^\vee_i=E^\vee_i,
\qquad  \quad E^*_iE^\vee_i E^*_i=E^*_i, \label{eq:efe}
\\
&E^\vee_iE_{d-i} E^\vee_i=E^\vee_i, \qquad E_i E^\vee_{d-i} E_i=E_i. \label{eq:fef}
\end{align}

\begin{lemma}\label{lem:del} The maps
\begin{align*}
\Psi = \sum_{i=0}^d E^\vee_{d-i}E_i, \qquad \qquad \Psi^*= \sum_{i=0}^d E^\vee_i E^*_i
\end{align*}
are invertible, and 
\begin{align*}
\Psi^{-1} = \sum_{i=0}^d E_{d-i} E^\vee_i, \qquad \qquad (\Psi^*)^{-1} = \sum_{i=0}^d E^*_iE^\vee_i.
\end{align*}
Moreover,
\begin{align*}
\Psi (E_iV)=U_{d-i}, \qquad \qquad \Psi^* (E^*_iV) =U_i \qquad \qquad (0 \leq i \leq d).
\end{align*}
\end{lemma} 
\begin{proof}  First we verify the assertions about $\Psi$.
Define $\Psi'= \sum_{i=0}^d E_{d-i}E^\vee_i$. Using \eqref{eq:efe}, \eqref{eq:fef} we routinely obtain $\Psi \Psi' = I$ and $\Psi' \Psi = I$.
Therefore, $\Psi$ is invertible and $\Psi^{-1} = \Psi'$. Pick an integer $i$ $(0 \leq i \leq d)$. We have
\begin{align*}
\Psi (E_iV) = \Biggl( \sum_{j=0}^d E^\vee_{d-j} E_j \Biggr) E_iV = E^\vee_{d-i} E_iV \subseteq E^\vee_{d-i} V = U_{d-i},
\end{align*}
and also
\begin{align*}
\Psi^{-1} (U_{d-i}) = \Biggl( \sum_{j=0}^d E_{d-j} E^\vee_{j} \Biggr) U_{d-i} =E_i E^\vee_{d-i} U_{d-i} \subseteq E_iV.
\end{align*}
 By these comments $\Psi(E_iV)=U_{d-i}$. We have verified the assertions about $\Psi$.
The assertions about $\Psi^*$ are similarly verified.
\end{proof}


\begin{lemma}\label{lem:uee}   {\rm (See \cite[Corollary~5.7]{TD00}.)} For $0 \leq i \leq d$ the subspaces $E_iV$, $E^*_iV$, $U_i$ have the same dimension.
Denoting this common dimension by $\rho_i$, we have $\rho_i =\rho_{d-i}$.
\end{lemma} 
\begin{proof} By Lemma \ref{lem:del}, the subspaces $E_iV$, $U_{d-i}$ have the same dimension and the subspaces $E^*_iV$, $U_i$ have the same
dimension. Replacing $i$ by $d-i$, we find that  $E_{d-i}V$, $U_i$ have the same dimension and $E^*_{d-i}V$, $U_{d-i}$ have the same dimension.
By these comments, the subspaces $E_iV$, $E^*_{d-i}V$ have the same dimension. 
Applying this result to $\Phi^\downarrow$, we find
that $E_iV$, $E^*_iV$ have the same dimension. 
The result follows.
\end{proof}

\noindent We mention a fact about the dimensions $\lbrace \rho_i\rbrace_{i=0}^d $ from Lemma \ref{lem:uee}. We will not use this fact, so we will not dwell on the proof.
\begin{lemma} {\rm (See \cite[Corollary~6.6]{TD00}.)} Referring to Lemma \ref{lem:uee}, we have $\rho_{i-1} \leq \rho_i$ for $1 \leq i \leq d/2$.
\end{lemma}

\begin{definition}\label{def:PhShape} \rm By the {\it shape of $\Phi$} we mean the sequence $\lbrace \rho_i \rbrace_{i=0}^d$ from Lemma \ref{lem:uee}.
\end{definition}

\noindent A TD system of shape $(1,1,\ldots, 1)$ is often called  a Leonard system \cite[Definition~1.4]{LS99}, \cite[Lemma~2.2]{qSerre}.
\medskip

\section{Decompositions and flags}
We continue to discuss the nontrivial TD system $ \Phi=(A;\lbrace E_i\rbrace_{i=0}^d;A^*;\lbrace E^*_i\rbrace_{i=0}^d)$ on $V$. In the previous
section we used $\Phi$ to obtain a sequence $\lbrace U_i \rbrace_{i=0}^d$ of subspaces of $V$. An analogous sequence is obtained from $\Phi^\downarrow$,
$\Phi^\Downarrow$, $\Phi^{\downarrow \Downarrow}$.
In this section we interpret the resulting sequences in a comprehensive way.

\medskip
\noindent First we define our terms. By a {\it decomposition of $V$} we mean a sequence  $\lbrace \mathcal V_i \rbrace_{i=0}^d$ of nonzero subspaces whose direct sum is equal to $V$.
Let $\lbrace \mathcal V_i \rbrace_{i=0}^d$ denote a decomposition of $V$.
For $0 \leq i \leq d$ we call $\mathcal V_i$ the {\it $i^{\rm th}$ component} of 
the decomposition.
For notational convenience, define $\mathcal V_{-1}=0$ and $\mathcal V_{d+1}=0$. 

\begin{example} \label{ex:decomp} \rm Each of the following sequences is a decomposition of $V$:
\begin{enumerate}
\item[\rm (i)] 
$\lbrace E_iV\rbrace_{i=0}^d$;
\item[\rm (ii)]  $\lbrace E^*_iV\rbrace_{i=0}^d$;
\item[\rm (iii)] the sequence $\lbrace U_i\rbrace_{i=0}^d$ from \eqref{eq:abbrev}.
\end{enumerate}
\end{example}

\noindent Let $\lbrace \mathcal V_i \rbrace_{i=0}^d$ denote a decomposition of $V$. For $0 \leq i \leq d$ define $s_i = {\rm dim}\,\mathcal V_i$, and note that $s_i \geq 1$. By construction $\sum_{i=0}^d s_i = {\rm dim}\,V$.
We call the sequence $\lbrace s_i \rbrace_{i=0}^d$ the {\it shape} of the decomposition $\lbrace \mathcal V_i \rbrace_{i=0}^d$.
\medskip

\noindent Next we recall the notion of a flag. 
Let $\lbrace s_i \rbrace_{i=0}^d$ denote a sequence of positive integers whose sum is equal to the dimension of $V$.
By a {\it flag on $V$ of shape $\lbrace s_i \rbrace_{i=0}^d$} we mean
a sequence $\lbrace F_i \rbrace_{i=0}^d$ of
subspaces of $V$ such that both (i) $F_{i-1}\subseteq F_i$ for $1 \leq i \leq d$;
(ii) $F_i$ has dimension $s_0+ s_1 + \cdots + s_i$ for $0 \leq i \leq d$. 
Let $\lbrace F_i\rbrace_{i=0}^d$ denote a flag on $V$. For $0 \leq i \leq d$ we call $F_i$ the {\it $i^{\rm th}$ component} of 
the flag. Note that $F_d=V$. For notational convenience, define $F_{-1}=0$ and $F_{d+1}=V$.
\medskip

\noindent
The following construction yields a flag on $V$. Let $\lbrace \mathcal V_i \rbrace_{i=0}^d$ denote a decomposition of $V$.
Define $F_i = \mathcal V_0+ \mathcal V_1+ \cdots + \mathcal V_i$ for $0 \leq i \leq d$. Then the sequence $\lbrace F_i \rbrace_{i=0}^d$ is a flag on $V$. The shape of this flag is equal to the shape of the decomposition $\lbrace \mathcal V_i \rbrace_{i=0}^d$.
We say that the flag $\lbrace F_i \rbrace_{i=0}^d$ is {\it induced} by the decomposition $\lbrace \mathcal V_i \rbrace_{i=0}^d$. 
\medskip

\noindent 
Next, we recall what it means for two flags to be opposite.
Let $\lbrace F_i \rbrace_{i=0}^d$ and 
$\lbrace F'_i \rbrace_{i=0}^d$ denote flags on $V$. These flags are called {\it opposite} whenever there exists a decomposition $\lbrace \mathcal V_i \rbrace_{i=0}^d$ of $V$ such that
$\lbrace \mathcal V_i \rbrace_{i=0}^d $
induces $\lbrace F_i\rbrace_{i=0}^d$  and  $\lbrace \mathcal V_{d-i}\rbrace_{i=0}^d$ induces $\lbrace F'_i \rbrace_{i=0}^d$. In this case $\mathcal V_i = F_i \cap F'_{d-i}$
for $0 \leq i \leq d$.
\medskip

\noindent We now use $\Phi$ to construct four flags on $V$. To keep track of these flags, we give each one a name. Let $\Omega$ denote the set
consisting of the four symbols
$0, D, 0^*, D^*$. Each flag 
gets a name $\lbrack u \rbrack$ with $u\in \Omega$.

\begin{definition}\rm
\label{lem:fourflag} 
In each row of the table below, we display a flag on $V$.
\bigskip

\centerline{
\begin{tabular}[t]{c|c}
       {\rm flag name} & {\rm $i^{\rm th}$ component of the flag }
 \\ \hline  \hline
	$\lbrack 0\rbrack$ & $E_0V+ E_1V+ \cdots + E_iV$    
	\\
	$\lbrack D\rbrack$ & $E_dV+ E_{d-1}V+ \cdots + E_{d-i}V$   \\
	$\lbrack 0^*\rbrack$ &
	$E^*_0V+E^*_1V+ \cdots + E^*_iV$  \\ 
	$\lbrack D^*\rbrack$ & 
	$E^*_{d}V+E^*_{d-1}V + \cdots +E^*_{d-i}V$
	\end{tabular}}
\medskip

\end{definition}

\noindent Referring to Definition \ref{lem:fourflag}, by construction the flags $\lbrack 0 \rbrack$, $\lbrack D \rbrack$ are opposite, and the flags  $\lbrack 0^* \rbrack$, $\lbrack D^* \rbrack$ are opposite.
In fact we have the following.
\begin{lemma} \label{lem:mop} 
The four flags in Definition
\ref{lem:fourflag} are mutually opposite.
\end{lemma}
\begin{proof} The flags $\lbrack 0 \rbrack$, $\lbrack D \rbrack$ are opposite, and the flags  $\lbrack 0^* \rbrack$, $\lbrack D^* \rbrack$ are opposite. To see that the flags $\lbrack 0 \rbrack$, $\lbrack 0^* \rbrack$ are opposite,
consider the decomposition $\lbrace U_i \rbrace_{i=0}^d $ of $V$ from \eqref{eq:abbrev}. By Lemma  \ref{lem:split}, $\lbrace U_i \rbrace_{i=0}^d$ induces $\lbrack 0^*\rbrack$ and $\lbrace U_{d-i} \rbrace_{i=0}^d $
induces $\lbrack 0 \rbrack$. Therefore the flags $\lbrack 0 \rbrack$, $\lbrack 0^*\rbrack$ are opposite. To finish the proof, apply this result to $\Phi^\downarrow$, $\Phi^\Downarrow$, $\Phi^{\downarrow \Downarrow}$.
\end{proof}

\medskip
\noindent Next, we use the four flags in Definition \ref{lem:fourflag} to construct six decompositions of $V$.
Pick distinct $u, v \in \Omega$. By Lemma \ref{lem:mop},  the flags $\lbrack u\rbrack$ and $\lbrack v \rbrack$ are opposite. Therefore, there exists a decomposition of $V$ that induces $\lbrack u \rbrack$ and whose inversion induces $\lbrack v \rbrack$.
This decomposition is unique, and denoted by $\lbrack u,v\rbrack$. Observe that the decomposition $\lbrack v, u\rbrack$ is the inversion of the decomposition $\lbrack u,v\rbrack$.
Up to inversion,
the above construction yields six decompositions of $V$. These decompositions are described in the following example.

\newpage
\begin{example} \rm
\label{thm:sixdecp} {\rm  (See \cite[Lemma~4.2]{tdanduqsl2hat}.)}
In each row of the table below, we display a decomposition of $V$.
\bigskip

\centerline{
\begin{tabular}[t]{c|c}
       {\rm decomp. name} & {\rm $i^{\rm th}$ component of the decomposition} 
 \\ \hline  \hline
	$\lbrack 0,D\rbrack$ & $E_iV$    
	\\
	$\lbrack 0^*,D^*\rbrack$ & $E^*_iV$   \\
	$\lbrack 0^*,0\rbrack$ & 
	$(E^*_{0}V+\cdots +E^*_iV) \cap
	(E_0V+\cdots +E_{d-i}V)$
	\\ 
	$\lbrack 0^*,D\rbrack$ &
	$(E^*_0V+\cdots + E^*_iV)\cap (E_iV+\cdots +E_dV)$   \\ 
        $\lbrack D^*,0\rbrack$ &
	$
	(E^*_{d-i}V+\cdots +E^*_dV) \cap
	(E_0V+\cdots + E_{d-i}V)$
	\\ 
	$ \lbrack D^*,D\rbrack $ & 
	$
	(E^*_{d-i}V+\cdots +E^*_dV)
	\cap
	(E_iV+\cdots +E_dV)
	$ 
	\end{tabular}}
\medskip
\noindent 
\end{example}

\begin{note} \label{note:u} \rm The decomposition $\lbrack 0^*,0\rbrack $ from Example \ref{thm:sixdecp} is the same as the decomposition $\lbrace U_i\rbrace_{i=0}^d$ from \eqref{eq:abbrev}.
\end{note}

\noindent  In the next result, we clarify how the six decompositions from
Example
\ref{thm:sixdecp} are related to the four flags from Definition \ref{lem:fourflag}.

\begin{lemma}
\label{thm:decsum} {\rm \rm (See \cite[Lemma~4.3]{tdanduqsl2hat}.)}
Let $\lbrace \mathcal V_i\rbrace_{i=0}^d$ 
denote a decomposition
of $V$ from
Example \ref{thm:sixdecp}.
Then for $0 \leq i \leq d$
the sums $\mathcal V_0+\cdots + \mathcal V_i$ and $\mathcal V_i+\cdots + \mathcal V_d$
are given in the table below.
\bigskip

\centerline{
\begin{tabular}[t]{c|c|c}
      {\rm decomp. name} &$\mathcal V_0+\cdots + \mathcal V_i$ & $\mathcal V_i+\cdots + \mathcal V_d$ \\ \hline  \hline
	$\lbrack 0,D\rbrack $ &
	$E_0V + \cdots + E_iV$  & $E_iV + \cdots + E_dV $  
\\
$\lbrack 0^*,D^*\rbrack $ &
	$E^*_0V+\cdots + E^*_iV$  & $ E^*_iV + \cdots + E^*_dV$   \\
      $\lbrack 0^*,0\rbrack $ & 
	$E^*_0V+\cdots +E^*_iV$  
	&
	$E_0V+\cdots +E_{d-i}V$
	\\ 
       $\lbrack 0^*,D\rbrack $ &
        $E^*_0V+\cdots + E^*_iV$ & $E_iV+\cdots +E_dV$   \\ 
       $\lbrack D^*,0\rbrack $ & 
	$E^*_{d-i}V+\cdots + E^*_dV$ 
&	
	$E_0V + \cdots +E_{d-i}V $ 
	\\ 
	$\lbrack D^*,D\rbrack $ &
	$E^*_{d-i}V+\cdots +E^*_dV$ 
&	
	$E_iV+\cdots +E_dV $
	\end{tabular}}
\medskip

\end{lemma}
\begin{proof} By the discussion above Example \ref{thm:sixdecp}.
\end{proof}

\noindent Recall the shape $\lbrace \rho_i \rbrace_{i=0}^d$ of $\Phi$ from Definition \ref{def:PhShape}.

\begin{lemma} {\rm (See \cite[Lemma~4.4]{tdanduqsl2hat}.)}
\label{lem:shape}
Each flag in Definition \ref{lem:fourflag}
has shape $\lbrace \rho_i \rbrace_{i=0}^d$. Each decomposition in
Example \ref{thm:sixdecp}
has shape $\lbrace \rho_i \rbrace_{i=0}^d$.
\end{lemma}
\begin{proof} By Lemma \ref{lem:uee} and the construction.
\end{proof}

\noindent Next, we
 describe the actions of
$A$ and $A^*$ on each of the six decompositions from Example \ref{thm:sixdecp}.

\begin{lemma} {\rm (See \cite[Lemma~5.1]{tdanduqsl2hat}.)}
\label{thm:aaction}
Let $\lbrace \mathcal V_i \rbrace_{i=0}^d$ 
denote a decomposition
of $V$ from
Example \ref{thm:sixdecp}.
Then for $0 \leq i \leq d$
the actions of $A$ and $A^*$ on $\mathcal V_i$ are described in the table below.

\bigskip

\centerline{
\begin{tabular}[t]{c|c|c}
      {\rm decomp. name} &{\rm action of $A$ on $\mathcal V_i$} & {\rm action of $A^*$ on $\mathcal V_i$}
      \\ \hline  \hline
	$\lbrack 0,D\rbrack $ &
	$ (A-\theta_iI)\mathcal V_i=0$  & 
	$A^* \mathcal V_i \subseteq \mathcal V_{i-1}+ \mathcal V_i+\mathcal V_{i+1}$ 
\\	
	$\lbrack 0^*,D^*\rbrack $ &
	$A \mathcal V_i \subseteq \mathcal V_{i-1}+ \mathcal V_i+\mathcal V_{i+1}$  &
	$ (A^*-\theta^*_iI)\mathcal V_i=0$   \\
      $\lbrack 0^*,0\rbrack $ & 
	$(A-\theta_{d-i}I)\mathcal V_i\subseteq \mathcal V_{i+1}$ &
	$(A^*-\theta^*_iI)\mathcal V_i \subseteq \mathcal V_{i-1}$   \\ 
       $\lbrack 0^*,D\rbrack $ &
        $(A-\theta_iI)\mathcal V_i\subseteq \mathcal V_{i+1}$ &
	$(A^*-\theta^*_iI)\mathcal V_i \subseteq \mathcal V_{i-1}$   \\ 
       $\lbrack D^*,0\rbrack $ & 
	$(A-\theta_{d-i}I)\mathcal V_i\subseteq \mathcal V_{i+1}$ &
	$(A^*-\theta^*_{d-i}I)\mathcal V_i \subseteq \mathcal V_{i-1}$   \\ 
	$\lbrack D^*,D\rbrack $ &
	$(A-\theta_iI)\mathcal V_i\subseteq \mathcal V_{i+1}$ &
	$(A^*-\theta^*_{d-i}I)\mathcal V_i \subseteq \mathcal V_{i-1}$  
	\end{tabular}}
\medskip

\end{lemma}
\begin{proof} For the decompositions $\lbrack 0,D\rbrack $ and $\lbrack 0^*,D^*\rbrack $ the result holds by Definition \ref{def:tdp}.
For the decomposition  $\lbrack 0^*,0\rbrack $ the result holds by Lemma \ref{lem:RL} and Note \ref{note:u}. To get the result for the remaining decompositions,
apply
Lemma \ref{lem:RL} to  $\Phi^\downarrow$, $\Phi^\Downarrow$, $\Phi^{\downarrow \Downarrow}$.
\end{proof}

\noindent Here is another version of Lemma \ref{thm:aaction}.

\begin{corollary} 
\label{thm:aaction2}
Let $\lbrace \mathcal V_i \rbrace_{i=0}^d$ 
denote a decomposition
of $V$ from
Example \ref{thm:sixdecp}.
Then for $0 \leq i \leq d$
the actions of $A$ and $A^*$ on $\mathcal V_i$ are described in the table below.

\bigskip

\centerline{
\begin{tabular}[t]{c|c|c}
      {\rm decomp. name} &{\rm action of $A$ on $\mathcal V_i$} & {\rm action of $A^*$ on $\mathcal V_i$}
      \\ \hline  \hline
	$\lbrack 0,D\rbrack $ &
	$ A \mathcal V_i\subseteq \mathcal V_i$  & 
	$A^* \mathcal V_i \subseteq \mathcal V_{i-1}+ \mathcal V_i+\mathcal V_{i+1}$ 
\\	
	$\lbrack 0^*,D^*\rbrack $ &
	$A \mathcal V_i \subseteq \mathcal V_{i-1}+ \mathcal V_i+\mathcal V_{i+1}$  &
	$ A^* \mathcal V_i\subseteq \mathcal V_i$   \\
       $\lbrack 0^*,0\rbrack $ & 
	$A \mathcal V_i\subseteq \mathcal V_i + \mathcal V_{i+1}$ &
	$A^*\mathcal V_i \subseteq \mathcal V_{i-1}+ \mathcal V_i$   \\ 
       $\lbrack 0^*,D\rbrack $ &
        $A \mathcal V_i\subseteq \mathcal V_i + \mathcal V_{i+1}$ &
	$A^* \mathcal V_i \subseteq \mathcal V_{i-1} + \mathcal V_i$   \\ 
       $\lbrack D^*,0\rbrack $ & 
	$A\mathcal V_i\subseteq \mathcal V_i + \mathcal V_{i+1}$ &
	$A^*\mathcal V_i \subseteq \mathcal V_{i-1}+ \mathcal V_i$   \\ 
	$\lbrack D^*,D\rbrack $ &
	$ A \mathcal V_i\subseteq \mathcal V_i + \mathcal V_{i+1}$ &
	$A^* \mathcal V_i \subseteq \mathcal V_{i-1}+\mathcal V_i$  
	\end{tabular}}
\medskip

\end{corollary}
\begin{proof} By Lemma \ref{thm:aaction}.
\end{proof}


\section{The tetrahedron diagram}
We continue to discuss the nontrivial TD system $ \Phi=(A;\lbrace E_i\rbrace_{i=0}^d;A^*;\lbrace E^*_i\rbrace_{i=0}^d)$ on $V$.
In the previous  section we used $\Phi$ to construct six decompositions of $V$. In this section we draw a diagram that illustrates how these six decompositions are related.
\medskip

\noindent  Let  $\lbrace \mathcal V_i \rbrace_{i=0}^d$ denote a decomposition of $V$.
We describe this decomposition by the diagram
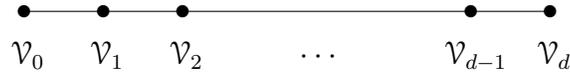
\begin{center}
\begin{picture}(0,20)
\put(-100,0){\line(1,0){200}}
\put(-102,-3){$\bullet$}	
\put(-104,-20){$\mathcal V_0$}
\put(-72,-3){$\bullet$}	
\put(-74,-20){$\mathcal V_1$}
\put(-42,-3){$\bullet$}
\put(-44,-20){$\mathcal V_2$}
\put(5,-20) {$\cdots $} 
	\put(67,-3){$\bullet$}	
	\put(60,-20){$\mathcal V_{d-1}$}
	
\put(97,-3){$\bullet$}	
\put(95,-20){$\mathcal V_d$}
\end{picture}
\end{center}
\vspace{1cm}

\noindent  The labels $\mathcal V_i$ might be suppressed, if they are clear from the context.
\medskip

\noindent Let $\lbrace \mathcal V_i\rbrace_{i=0}^d$ and 
 $\lbrace \mathcal V'_i\rbrace_{i=0}^d$ denote decompositions of $V$.
The condition
\begin{align*} 
\mathcal V_0 + \mathcal V_1 + \cdots + \mathcal V_i = \mathcal V'_0 + \mathcal V'_1 + \cdots + \mathcal V'_i \qquad  (0 \leq i \leq d)
\end{align*}
will be described by the diagram
\medskip

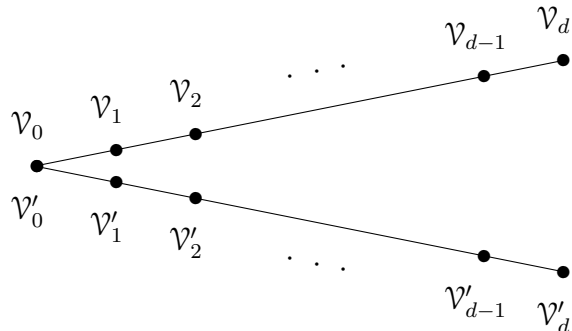
\begin{center}
\begin{picture}(0,60)
\put(-100,0){\line(5,1){200}}
\put(-100,0){\line(5,-1){200}}

\put(-102,-3){$\bullet$}	
\put(-72,-9){$\bullet$}	
\put(-42,-15){$\bullet$}	
\put(67,-37){$\bullet$}	
\put(97,-43){$\bullet$}

\put(-109,-20){$\mathcal V'_0$}
\put(-79,-26){$\mathcal V'_1$}
\put(-49,-32){$\mathcal V'_2$}
\put(-5,-38) {$\cdot$}	
\put(5,-40) {$\cdot$}	
\put(15,-42) {$\cdot$}	
\put(55,-54){$\mathcal V'_{d-1}$}
\put(90,-60){$\mathcal V'_d$}

\put(-102,-3){$\bullet$}	
\put(-72,3){$\bullet$}	
\put(-42,9){$\bullet$}	
\put(67,31){$\bullet$}	
\put(97,37){$\bullet$}

\put(-109,12){$\mathcal V_0$}
\put(-79,18){$\mathcal V_1$}
\put(-49,24){$\mathcal V_2$}
\put(-5,31) {$\cdot$}	
\put(5,33) {$\cdot$}	
\put(15,35) {$\cdot$}	
\put(55,46){$\mathcal V_{d-1}$}
\put(90,52){$\mathcal V_d$}

\end{picture}
\end{center}
\vspace{2cm}

\noindent To illustrate the above diagram convention, consider the decomposition $\lbrace U_i \rbrace_{i=0}^d$ of $V$ from \eqref{eq:abbrev}. By Lemma \ref{lem:split} we have
\begin{align*}
&U_0 + \cdots + U_i = E^*_0V+\cdots + E^*_iV,
\qquad \qquad U_i + \cdots + U_d = E_0V+\cdots + E_{d-i}V
\end{align*}
 for $0 \leq i \leq d$.
The corresponding diagram is shown below:

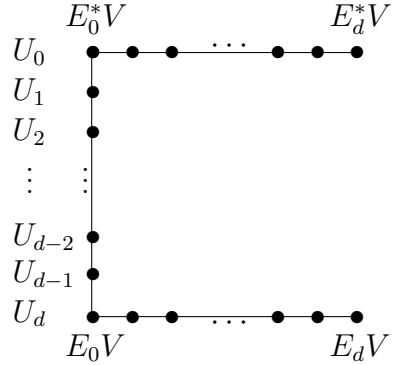
\begin{center}

\begin{picture}(100,80)

\put(-50,50){\line(1,0){100}}


\put(50,-50){\line(-1,0){100}}

\put(-50,-50){\line(0,1){100}}

\put(-52.5,47){$\bullet$}
\put(-37.5,47){$\bullet$}
\put(-22.5,47){$\bullet$}

\put(-5,50){$\cdots$}

\put(47.5,47){$\bullet$}
\put(32.5,47){$\bullet$}
\put(17.5,47){$\bullet$}

\put(-52.5,-53){$\bullet$}
\put(-37.5,-53){$\bullet$}
\put(-22.5,-53){$\bullet$}

\put(-5,-55){$\cdots$}

\put(47.5,-53){$\bullet$}
\put(32.5,-53){$\bullet$}
\put(17.5,-53){$\bullet$}

\put(-52.5,47){$\bullet$}
\put(-52.5,32){$\bullet$}
\put(-52.5,17){$\bullet$}
\put(-52.5,-23){$\bullet$}
\put(-52.5,-37){$\bullet$}
\put(-54,-3){$\vdots$}
\put(-75,-3){$\vdots$}

\put(-60,60){$E^*_0V$}
\put(40,60){$E^*_dV$}
\put(-60,-65){$E_0V$}
\put(40,-65){$E_dV$}

\put(-80,47){$U_0$}
\put(-80,32){$U_1$}
\put(-80,17){$U_2$}
\put(-80,-23){$U_{d-2}$}
\put(-80,-37){$U_{d-1}$}
\put(-80,-52){$U_d$}


\end{picture}
\end{center}
\vspace{2.5cm}
\noindent In Example \ref{thm:sixdecp} we gave six decompositions of $V$. The corresponding diagram is shown below:

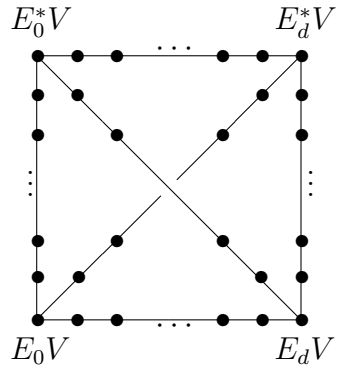
\begin{center}

\begin{picture}(100,80)

\put(-50,50){\line(1,0){50}}
\put(0,50){\line(1,0){50}}

\put(50,50){\line(0,-1){50}}
\put(50,0){\line(0,-1){50}}

\put(50,-50){\line(-1,0){50}}
\put(0,-50){\line(-1,0){50}}

\put(-50,-50){\line(0,1){50}}
\put(-50,0){\line(0,1){50}}

\put(0,0){\line(-1,1){25}}
\put(-25,25){\line(-1,1){25}}
\put(0,0){\line(1,-1){25}}
\put(25,-25){\line(1,-1){25}}

\put(3,3){\line(1,1){22}}
\put(25,25){\line(1,1){25}}
\put(-3,-3){\line(-1,-1){22}}
\put(-25,-25){\line(-1,-1){25}}

\put(-60,60){$E^*_0V$}
\put(40,60){$E^*_dV$}
\put(-60,-65){$E_0V$}
\put(40,-65){$E_dV$}

\put(-52.5,47){$\bullet$}
\put(-37.5,47){$\bullet$}
\put(-22.5,47){$\bullet$}

\put(-5,50){$\cdots$}

\put(47.5,47){$\bullet$}
\put(32.5,47){$\bullet$}
\put(17.5,47){$\bullet$}

\put(-52.5,-53){$\bullet$}
\put(-37.5,-53){$\bullet$}
\put(-22.5,-53){$\bullet$}

\put(-5,-55){$\cdots$}

\put(47.5,-53){$\bullet$}
\put(32.5,-53){$\bullet$}
\put(17.5,-53){$\bullet$}

\put(-52.5,32){$\bullet$}
\put(-52.5,17){$\bullet$}
\put(-52.5,-23){$\bullet$}
\put(-52.5,-37){$\bullet$}
\put(-54,-3){$\vdots$}
\put(52,-3){$\vdots$}

\put(47.5,32){$\bullet$}
\put(47.5,17){$\bullet$}
\put(47.5,-23){$\bullet$}
\put(47.5,-37){$\bullet$}

\put(-37.5,32){$\bullet$}
\put(-22.5,17){$\bullet$}
\put(17.5,-23){$\bullet$}
\put(32,-37){$\bullet$}

\put(32.5,32){$\bullet$}
\put(17.5,17){$\bullet$}
\put(-22.5,-23){$\bullet$}
\put(-37,-37){$\bullet$}

\end{picture}
\end{center}

\vspace{2.5cm}

\noindent This diagram is called the {\it tetrahedron diagram} of $\Phi$.
\medskip

\noindent Next we use the tetrahedron diagram to illustrate Corollary
\ref{thm:aaction2}.
The following picture shows how $A$ acts on the  decompositions of $V$ from the tetrahedron diagram, for $d=8$:

\begin{center}
\begin{picture}(100,80)

\put(-50,50){\line(1,0){50}}
\put(0,50){\line(1,0){50}}

\put(50,50){\line(0,-1){50}}
\put(50,0){\line(0,-1){50}}

\put(50,-50){\line(-1,0){50}}
\put(0,-50){\line(-1,0){50}}

\put(-50,-50){\line(0,1){50}}
\put(-50,0){\line(0,1){50}}

\put(0,0){\line(-1,1){25}}
\put(-25,25){\line(-1,1){25}}
\put(0,0){\line(1,-1){25}}
\put(25,-25){\line(1,-1){25}}

\put(3,3){\line(1,1){22}}
\put(25,25){\line(1,1){25}}
\put(-3,-3){\line(-1,-1){22}}
\put(-25,-25){\line(-1,-1){25}}

\put(-60,60){$E^*_0V$}
\put(40,60){$E^*_dV$}
\put(-60,-65){$E_0V$}
\put(40,-65){$E_dV$}
\put(-140,0){$A$ action:}
\put(-52.5,47){$\bullet$}
\put(-40,47){$\bullet$}
\put(-27.5,47){$\bullet$}
\put(-15,47){$\bullet$}
\put(-2.5,47){$\bullet$}

\put(47,47){$\bullet$}
\put(34.5,47){$\bullet$}
\put(22,47){$\bullet$}
\put(9.5,47){$\bullet$}

\put(-12.5,50){\circle{9}}
\put(0,50){\circle{9}}
\put(0,50){\circle{12}}
\put(12.5,50){\circle{9}}
\put(-52.5,-53){$\bullet$}
\put(-40,-53){$\bullet$}
\put(-27.5,-53){$\bullet$}
\put(-15,-53){$\bullet$}
\put(-2.5,-53){$\bullet$}

\put(47.5,-53){$\bullet$}
\put(35,-53){$\bullet$}
\put(22.5,-53){$\bullet$}
\put(10,-53){$\bullet$}

\put(0,-50){\circle{9}}
\put(0,-50){\circle{12}}
\put(-52.5,34.5){$\bullet$}
\put(-52.5,22){$\bullet$}
\put(-52.5,9.5){$\bullet$}
\put(-52.5,-3){$\bullet$}
\put(-52.5,-15.5){$\bullet$}
\put(-52.5,-28){$\bullet$}
\put(-52.5,-40.5){$\bullet$}

\put(-50,12.5){\circle{9}}
\put(-50,12.5){\circle{12}}
\put(-50,0){\circle{9}}
\put(50,12.5){\circle{9}}
\put(50,12.5){\circle{12}}
\put(50,0){\circle{9}}

\put(-25,-25){\circle{9}}
\put(-12.5,-12.5){\circle{12}}
\put(-12.5,-12.5){\circle{9}}
\put(25,-25){\circle{9}}
\put(12.5,-12.5){\circle{12}}
\put(12.5,-12.5){\circle{9}}

\put(47.5,34.5){$\bullet$}
\put(47.5,22){$\bullet$}
\put(47.5,9.5){$\bullet$}
\put(47.5,-3){$\bullet$}
\put(47.5,-15.5){$\bullet$}
\put(47.5,-28){$\bullet$}
\put(47.5,-40.5){$\bullet$}

\put(-40,34.5){$\bullet$}
\put(-27.5,22){$\bullet$}
\put(-15,9.5){$\bullet$}
\put(-2.5,-3){$\bullet$}

\put(35,-40.5){$\bullet$}
\put(22.5,-28){$\bullet$}
\put(10,-15.5){$\bullet$}

\put(35,34.5){$\bullet$}
\put(22.5,22){$\bullet$}
\put(10,9.5){$\bullet$}
\put(-15,-15.5){$\bullet$}
\put(-27.5,-28){$\bullet$}
\put(-40,-40.5){$\bullet$}

\end{picture}
\end{center}

\vspace{2.5cm}
\newpage
\noindent The following picture shows how $A^*$ acts on the  decompositions of $V$ from the tetrahedron diagram, for $d=8$:
\begin{center}

\begin{picture}(100,80)

\put(-50,50){\line(1,0){50}}
\put(0,50){\line(1,0){50}}

\put(50,50){\line(0,-1){50}}
\put(50,0){\line(0,-1){50}}

\put(50,-50){\line(-1,0){50}}
\put(0,-50){\line(-1,0){50}}

\put(-50,-50){\line(0,1){50}}
\put(-50,0){\line(0,1){50}}

\put(0,0){\line(-1,1){25}}
\put(-25,25){\line(-1,1){25}}
\put(0,0){\line(1,-1){25}}
\put(25,-25){\line(1,-1){25}}

\put(3,3){\line(1,1){22}}
\put(25,25){\line(1,1){25}}
\put(-3,-3){\line(-1,-1){22}}
\put(-25,-25){\line(-1,-1){25}}

\put(-60,60){$E^*_0V$}
\put(40,60){$E^*_dV$}
\put(-60,-65){$E_0V$}
\put(40,-65){$E_dV$}
\put(-140,0){$A^*$ action:}
\put(-52.5,47){$\bullet$}
\put(-40,47){$\bullet$}
\put(-27.5,47){$\bullet$}
\put(-15,47){$\bullet$}
\put(-2.5,47){$\bullet$}

\put(47,47){$\bullet$}
\put(34.5,47){$\bullet$}
\put(22,47){$\bullet$}
\put(9.5,47){$\bullet$}

\put(-12.5,-50){\circle{9}}
\put(0,-50){\circle{9}}
\put(0,-50){\circle{12}}
\put(12.5,-50){\circle{9}}
\put(-52.5,-53){$\bullet$}
\put(-40,-53){$\bullet$}
\put(-27.5,-53){$\bullet$}
\put(-15,-53){$\bullet$}
\put(-2.5,-53){$\bullet$}

\put(47.5,-53){$\bullet$}
\put(35,-53){$\bullet$}
\put(22.5,-53){$\bullet$}
\put(10,-53){$\bullet$}

\put(0,50){\circle{9}}
\put(0,50){\circle{12}}
\put(-52.5,34.5){$\bullet$}
\put(-52.5,22){$\bullet$}
\put(-52.5,9.5){$\bullet$}
\put(-52.5,-3){$\bullet$}
\put(-52.5,-15.5){$\bullet$}
\put(-52.5,-28){$\bullet$}
\put(-52.5,-40.5){$\bullet$}

\put(-50,-12.5){\circle{9}}
\put(-50,-12.5){\circle{12}}
\put(-50,0){\circle{9}}
\put(50,-12.5){\circle{9}}
\put(50,-12.5){\circle{12}}
\put(50,0){\circle{9}}

\put(-25,25){\circle{9}}
\put(-12.5,12.5){\circle{12}}
\put(-12.5,12.5){\circle{9}}
\put(25,25){\circle{9}}
\put(12.5,12.5){\circle{12}}
\put(12.5,12.5){\circle{9}}

\put(47.5,34.5){$\bullet$}
\put(47.5,22){$\bullet$}
\put(47.5,9.5){$\bullet$}
\put(47.5,-3){$\bullet$}
\put(47.5,-15.5){$\bullet$}
\put(47.5,-28){$\bullet$}
\put(47.5,-40.5){$\bullet$}

\put(-40,34.5){$\bullet$}
\put(-27.5,22){$\bullet$}
\put(-15,9.5){$\bullet$}
\put(-2.5,-3){$\bullet$}

\put(35,-40.5){$\bullet$}
\put(22.5,-28){$\bullet$}
\put(10,-15.5){$\bullet$}

\put(35,34.5){$\bullet$}
\put(22.5,22){$\bullet$}
\put(10,9.5){$\bullet$}
\put(-15,-15.5){$\bullet$}
\put(-27.5,-28){$\bullet$}
\put(-40,-40.5){$\bullet$}

\end{picture}
\end{center}
\vspace{2.5cm}

\noindent We will return to the tetrahedron diagram in Section 9.
\section{Some comments about flags and decompositions}

We continue to discuss the nontrivial TD system $ \Phi=(A;\lbrace E_i\rbrace_{i=0}^d;A^*;\lbrace E^*_i\rbrace_{i=0}^d)$ on $V$.
In Section 4  we used $\Phi$ to construct six decompositions of $V$. In Lemma \ref{thm:aaction} and Corollary \ref{thm:aaction2}
 we described how
$A$ and $A^*$ act on these six decompositions. Shortly we will consider some additional  maps in ${\rm End}(V)$, and describe how these maps
act on the six decompositions. To prepare for this description, we have some general comments  about how maps in ${\rm End}(V)$ act on flags and decompositions.

\begin{definition}\rm Let $\lbrace F_i \rbrace_{i=0}^d$ denote a flag on $V$, and let $\psi \in {\rm End}(V)$. We say that
$\psi$ 
{\it stabilizes} $\lbrace F_i \rbrace_{i=0}^d$  whenever $\psi F_i \subseteq F_i$ for $0 \leq i \leq d$. We say that
$\psi$ {\it raises} $\lbrace F_i \rbrace_{i=0}^d$ whenever $\psi F_i \subseteq F_{i+1}$ for $0 \leq i \leq d-1$.
\end{definition}

\noindent For the rest of this section, the following assumption is in effect.
\begin{assumption} \label{lem:ass} \rm
Let $\lbrace \mathcal V_i \rbrace_{i=0}^d $ denote a decomposition of $V$. Let $\lbrace F_i\rbrace_{i=0}^d$ denote the flag on $V$ induced by $\lbrace \mathcal V_i\rbrace_{i=0}^d$, 
and let $\lbrace F'_i \rbrace_{i=0}^d$ denote the flag on $V$ induced by $\lbrace \mathcal V_{d-i} \rbrace_{i=0}^d$. Thus
\begin{align*}
F_i = \mathcal V_0 + \cdots + \mathcal V_i, \qquad \qquad  F'_i = \mathcal V_d  + \cdots + \mathcal V_{d-i}
\end{align*}
\noindent for $0 \leq i \leq d$.
\end{assumption}

\begin{lemma} \label{lem:RR} With reference to Assumption \ref{lem:ass}, the following are equivalent for $\psi \in {\rm End}(V)$:
\begin{enumerate}
\item[\rm (i)] $\psi$ raises $\lbrace F_i \rbrace_{i=0}^d$ and $\lbrace F'_i \rbrace_{i=0}^d$;
\item[\rm (ii)] $\psi \mathcal V_i \subseteq \mathcal V_{i-1} + \mathcal V_i + \mathcal V_{i+1} $ for $0 \leq i \leq d$.
\end{enumerate}
\end{lemma}
\begin{proof} ${\rm (i) \Rightarrow (ii)}$ We have
\begin{align*} 
\psi \mathcal V_i &\subseteq \psi (\mathcal V_0 + \cdots + \mathcal V_i) = \psi F_i \subseteq F_{i+1} = \mathcal V_0 + \cdots +\mathcal V_{i+1}
\end{align*}
and also
\begin{align*}
\psi \mathcal V_i &\subseteq \psi (\mathcal V_i + \cdots + \mathcal V_d) = \psi F'_{d-i}  \subseteq F'_{d-i+1} = \mathcal V_{i-1} + \cdots + \mathcal V_d.
\end{align*}
Therefore
\begin{align*}
\psi \mathcal V_i &\subseteq (\mathcal V_0 + \cdots + \mathcal V_{i+1})\cap (\mathcal V_{i-1} + \cdots + \mathcal V_d) =\mathcal V_{i-1} + \mathcal V_i + \mathcal V_{i+1}.
\end{align*}
${\rm (ii) \Rightarrow (i)}$ For $0 \leq i \leq d-1$ we have
\begin{align*}
\psi F_i = \psi (\mathcal V_0 + \cdots + \mathcal V_i) \subseteq \mathcal V_0 + \cdots + \mathcal V_{i+1} = F_{i+1}
\end{align*}
and also
\begin{align*}
\psi F'_i = \psi (\mathcal V_d + \cdots + \mathcal V_{d-i}) \subseteq \mathcal V_d + \cdots + \mathcal V_{d-i-1} = F'_{i+1}.
\end{align*}
\end{proof}

\begin{lemma} \label{lem:RS} With reference to Assumption \ref{lem:ass}, the following are equivalent for $\psi \in {\rm End}(V)$:
\begin{enumerate}
\item[\rm (i)] $\psi$ raises $\lbrace F_i \rbrace_{i=0}^d$ and stabilizes $\lbrace F'_i \rbrace_{i=0}^d$;
\item[\rm (ii)] $\psi \mathcal V_i \subseteq  \mathcal V_i + \mathcal V_{i+1} $ for $0 \leq i \leq d$.
\end{enumerate}
\end{lemma}
\begin{proof} ${\rm (i) \Rightarrow (ii)}$ We have
\begin{align*} 
\psi \mathcal V_i &\subseteq \psi (\mathcal V_0 + \cdots + \mathcal V_i) = \psi F_i \subseteq F_{i+1} = \mathcal V_0 + \cdots +\mathcal V_{i+1}
\end{align*}
and also
\begin{align*}
\psi \mathcal V_i &\subseteq \psi (\mathcal V_i + \cdots + \mathcal V_d) = \psi F'_{d-i}  \subseteq F'_{d-i} = \mathcal V_{i} + \cdots + \mathcal V_d.
\end{align*}
Therefore
\begin{align*}
\psi \mathcal V_i &\subseteq (\mathcal V_0 + \cdots + \mathcal V_{i+1})\cap (\mathcal V_{i} + \cdots + \mathcal V_d)  =\mathcal V_i + \mathcal V_{i+1}.
\end{align*}
${\rm (ii) \Rightarrow (i)}$ For $0 \leq i \leq d-1$ we have
\begin{align*}
\psi F_i = \psi (\mathcal V_0 + \cdots + \mathcal V_i) \subseteq \mathcal V_0 + \cdots + \mathcal V_{i+1} = F_{i+1}
\end{align*}
and also
\begin{align*}
\psi F'_i = \psi (\mathcal V_d + \cdots + \mathcal V_{d-i}) \subseteq \mathcal V_d + \cdots + \mathcal V_{d-i} = F'_{i}.
\end{align*}
\end{proof}

\begin{lemma} \label{lem:SR} With reference to Assumption \ref{lem:ass}, the following are equivalent for $\psi \in {\rm End}(V)$:
\begin{enumerate}
\item[\rm (i)] $\psi$ stabilizes  $\lbrace F_i \rbrace_{i=0}^d$ and raises $\lbrace F'_i \rbrace_{i=0}^d$;
\item[\rm (ii)] $\psi \mathcal V_i \subseteq \mathcal V_{i-1} + \mathcal V_i  $ for $0 \leq i \leq d$.
\end{enumerate}
\end{lemma}
\begin{proof} Apply Lemma \ref{lem:RS}
to the decomposition $\lbrace \mathcal V_{d-i} \rbrace_{i=0}^d$.
\end{proof}

\begin{lemma}\label{lem:SS}  With reference to Assumption \ref{lem:ass}, the following are equivalent for $\psi \in {\rm End}(V)$:
\begin{enumerate}
\item[\rm (i)] $\psi$ stabilizes $\lbrace F_i \rbrace_{i=0}^d$ and $\lbrace F'_i \rbrace_{i=0}^d$;
\item[\rm (ii)] $\psi \mathcal V_i \subseteq \mathcal V_i$ for $0 \leq i \leq d$.
\end{enumerate}
\end{lemma}
\begin{proof} ${\rm (i) \Rightarrow (ii)}$ We have
\begin{align*} 
\psi \mathcal V_i &\subseteq \psi (\mathcal V_0 + \cdots + \mathcal V_i) = \psi F_i \subseteq F_{i} = \mathcal V_0 + \cdots +\mathcal V_{i}
\end{align*}
and also
\begin{align*}
\psi \mathcal V_i &\subseteq \psi (\mathcal V_i + \cdots + \mathcal V_d) = \psi F'_{d-i}  \subseteq F'_{d-i} = \mathcal V_{i} + \cdots + \mathcal V_d.
\end{align*}
Therefore
\begin{align*}
\psi \mathcal V_i &\subseteq (\mathcal V_0 + \cdots + \mathcal V_{i})\cap (\mathcal V_{i} + \cdots + \mathcal V_d) =\mathcal V_i .
\end{align*}
${\rm (ii) \Rightarrow (i)}$ For $0 \leq i \leq d-1$ we have
\begin{align*}
\psi F_i = \psi (\mathcal V_0 + \cdots + \mathcal V_i) \subseteq \mathcal V_0 + \cdots + \mathcal V_{i} = F_{i}
\end{align*}
and also
\begin{align*}
\psi F'_i = \psi (\mathcal V_d + \cdots + \mathcal V_{d-i}) \subseteq \mathcal V_d + \cdots + \mathcal V_{d-i} = F'_{i}.
\end{align*}
\end{proof}

\section{The tridiagonal relations}
\noindent  We continue to discuss the nontrivial TD system $ \Phi=(A;\lbrace E_i\rbrace_{i=0}^d;A^*;\lbrace E^*_i\rbrace_{i=0}^d)$ on $V$. In this section we recall some relations satisfied by $A, A^*$.

\begin{lemma} {\rm (See \cite[Theorem~10.1]{TD00}.)}
\label{tdptheorem} 
There exists a sequence of scalars
$\beta,\gamma,\gamma^*,\varrho,\varrho^*$ in $\mathbb F$ such
that both
\begin{align}
\label{eq:TD1}
0 &= \lbrack A, A^2 A^*-\beta A A^* A+A^* A^2-\gamma\left(
A A^*+A^* A\right)-\varrho A^* \rbrack,
\\
\label{eq:TD2}
 0&=\lbrack A^*, A^{*2} A-\beta A^*AA^*
+AA^{*2}-\gamma^* \left(A^*A+AA^*\right)
-\varrho^* A\rbrack.
\end{align}
The sequence $\beta,\gamma,\gamma^*,\varrho,\varrho^*$ is unique if $d\geq 3$.
\end{lemma}
\noindent The relations \eqref{eq:TD1}, \eqref{eq:TD2}  are called the
{\it tridiagonal relations}. 
\medskip

\begin{lemma} {\rm (See  \cite[Theorem~4.3]{qSerre}.)}
\label{eq:lastchancedolangradyS992}
For the TD system $\Phi$ and scalars 
 $\beta, \gamma, \gamma^*,
\varrho, \varrho^*$ in $\mathbb F$,  these scalars satisfy \eqref{eq:TD1}, \eqref{eq:TD2} if and only if the following {\rm (i)--(v)} hold:
\begin{enumerate}
\item[\rm (i)] 
the expressions
\begin{align*}
\frac{\theta_{i-2}-\theta_{i+1}}{\theta_{i-1}-\theta_i},\qquad \qquad  
 \frac{\theta^*_{i-2}-\theta^*_{i+1}}{\theta^*_{i-1}-\theta^*_i} 
\end{align*} 
 are both equal to $\beta +1$ for $2\leq i \leq d-1$;
 \item[\rm (ii)]   
$\gamma = \theta_{i-1}-\beta \theta_i + \theta_{i+1} $ $(1 \leq i \leq d-1)$;
\item[\rm (iii)] 
$\gamma^* = \theta^*_{i-1}-\beta \theta^*_i + \theta^*_{i+1}$ $(1 \leq i \leq d-1)$;
\item[\rm (iv)] 
$\varrho = \theta^2_{i-1}-\beta \theta_{i-1}\theta_i+\theta_i^2-\gamma (\theta_{i-1}+\theta_i)$ $(1 \leq i \leq d)$;
\item[\rm (v)] 
$\varrho^*= \theta^{*2}_{i-1}-\beta \theta^*_{i-1}\theta^*_i+\theta_i^{*2}-
\gamma^* (\theta^*_{i-1}+\theta^*_i)$ $(1 \leq i \leq d)$.
\end{enumerate}
\end{lemma}

\noindent Shortly we will impose a restriction on the eigenvalues and dual eigenvalues of $\Phi$. In order to motivate this
restriction, we consider a certain algebra $U^+_q$. This algebra is discussed in the next section. 

\section{The algebra $U^+_q$ and its alternating elements}

\noindent{From now until the end of Section 10, we fix a nonzero scalar $q$ in the algebraic closure of $\mathbb F$, 
such that (i) $q$  is not a root of unity; (ii) $q^2 \in \mathbb F$. We abbreviate $b=q^2$. 
\noindent For elements $X, Y$ in any algebra, recall the notation
\begin{align*}
\lbrack X, Y \rbrack=  XY- YX, \qquad \qquad
\lbrack X, Y \rbrack_b= b XY- YX.
\end{align*}
\noindent Note that 
\begin{align*}
\lbrack X, \lbrack X, &\lbrack X, Y\rbrack_b \rbrack_{b^{-1}} \rbrack\\
&= 
X^3Y-(b+b^{-1}+1) X^2YX+ 
(b+b^{-1}+1)XYX^2 -YX^3.
\end{align*}
We will refer to the $q$-deformed enveloping algebra $U_q(\widehat{\mathfrak{sl}}_2)$}; see for example \cite{charp}.
\begin{definition}
\label{def:nposp}
\rm 
(See \cite[Corollary~3.2.6]{lusztig}.)
Define the algebra $U^+_q$ 
by generators $W_0, W_1$ and relations
\begin{align}
&
\lbrack W_0, \lbrack W_0, \lbrack W_0, W_1\rbrack_b \rbrack_{b^{-1}} \rbrack=0,
\qquad \qquad 
\lbrack W_1, \lbrack W_1, \lbrack W_1, W_0\rbrack_b \rbrack_{b^{-1}}
\rbrack=0.
\label{eq:nqSerre1}
\end{align}
\noindent We call $U^+_q$ the {\it positive part of 
$U_q(\widehat{\mathfrak{sl}}_2)$}.
The relations (\ref{eq:nqSerre1})
are called the {\it $q$-Serre relations}.
\end{definition}

\noindent
We will be discussing automorphisms
and antiautomorphisms.
  For an algebra $\mathcal A$,
an {\it automorphism} of $\mathcal A$ is 
an algebra isomorphism $\mathcal A \to \mathcal A$.
The {\it opposite algebra} $\mathcal A^{\rm opp}$ consists
of the vector space $\mathcal A$ and multiplication map
$\mathcal A \times \mathcal A \to \mathcal A$, $(a, b) \mapsto ba$.
An {\it antiautomorphism} of $\mathcal A$ 
is an
algebra isomorphism $\mathcal A \to \mathcal A^{\rm opp}$.
\begin{lemma}
\label{lem:nAAut} {\rm (See \cite[Lemma~2.3]{alternating}.)}
There exists an  automorphism $\sigma$ of
$U^+_q$ that swaps $W_0, W_1$. 
There exists an  antiautomorphism
$\dagger $ of $U^+_q$ that fixes each of $W_0$, $W_1$.
\end{lemma}

\begin{lemma} \label{lem:ms} For nonzero  $\lambda_0, \lambda_1 \in \mathbb F$ there exists an automorphism of $U^+_q$ that sends 
$W_0 \mapsto \lambda_0 W_0$ and $W_1 \mapsto \lambda_1 W_1$.
\end{lemma} 
\begin{proof} The $q$-Serre relations are homogeneous in $W_0$ and $W_1$.
\end{proof}

\noindent The alternating elements of
$U^+_q$ were introduced in \cite{alternating}. There are four kinds of alternating elements, denoted 
\begin{align}
\label{eq:nWWGGn}
\lbrace  W_{-k}\rbrace_{k \in \mathbb N}, \quad 
\lbrace  W_{k+1}\rbrace_{k \in \mathbb N}, \quad
\lbrace  G_{k+1}\rbrace_{k\in \mathbb N}, \quad
\lbrace {\tilde G}_{k+1}\rbrace_{k \in \mathbb N}.
\end{align}

\noindent 
These elements satisfy many relations; see
\cite[Proposition~5.7]{alternating},
\cite[Proposition~5.10]{alternating},
\cite[Proposition~5.11]{alternating},
\cite[Proposition~6.3]{alternating},
\cite[Proposition~8.1]{alternating}. Some of these relations are listed below.

\begin{lemma} 
\label{lem:nrel1}
{\rm (See \cite[Proposition~5.7]{alternating}.)}
For $k \in \mathbb N$
the following relations hold in $U^+_q$:
\begin{align*}
&
 \lbrack  W_0,  W_{k+1}\rbrack= 
\lbrack  W_{-k},  W_{1}\rbrack=
(1-b^{-1})({\tilde G}_{k+1} -  G_{k+1}),
\\
&
\lbrack  W_0,  G_{k+1}\rbrack_b= 
\lbrack {{\tilde G}}_{k+1},  W_{0}\rbrack_b= 
 (b-1)W_{-k-1},
\\
&
\lbrack G_{k+1},  W_{1}\rbrack_b= 
\lbrack  W_{1}, { {\tilde G}}_{k+1}\rbrack_b= 
(b-1) W_{k+2}.
\end{align*}
\end{lemma}

\begin{lemma}
\label{lem:nrel2}
{\rm (See \cite[Proposition~5.10]{alternating}.)}
For $k, \ell \in \mathbb N$ the
following relations hold in $U^+_q$:
\begin{align*}
&
\lbrack  W_{-k},  W_{-\ell}\rbrack=0,  \qquad 
\lbrack  W_{k+1},  W_{\ell+1}\rbrack= 0,
\\
&
\lbrack  G_{k+1},  G_{\ell+1}\rbrack=0,
\qquad 
\lbrack {\tilde G}_{k+1},  {\tilde G}_{\ell+1}\rbrack= 0.
\end{align*}
\end{lemma}

\noindent For notational convenience define $G_0=1$
and $\tilde G_0=1$.

\begin{lemma}
\label{prop:nGGWW}
{\rm (See \cite[Proposition~8.1]{alternating}.)}
For $n\geq 1$
the following relation holds in $U^+_q$:
\begin{align}
&
\sum_{k=0}^n  G_{k}  \tilde G_{n-k}b^{-k} 
= 
\sum_{k=0}^{n-1} W_{-k} W_{n-k} b^{-k}.
\label{eq:nGGWW1}
\end{align}
\end{lemma}

\begin{lemma}
\label{lem:nrecgen}
{\rm (See \cite[Proposition~8.2]{alternating}.)}
Using the equations below, the alternating elements in
$U^+_q$
are recursively obtained from $W_0, W_1$ in the following order:
\begin{align*}
W_0, \quad W_1, \quad G_1, \quad \tilde G_1, \quad W_{-1}, \quad W_2, \quad
 G_2, \quad \tilde G_2, \quad W_{-2}, \quad W_3, \quad \ldots
\end{align*}
For $n\geq 1$,
\begin{align}
G_n &= \frac{\sum_{k=0}^{n-1} W_{-k} W_{n-k} b^{-k} 
-
\sum_{k=1}^{n-1} G_k  \tilde G_{n-k}  b^{-k} }{1+b^{-n}}
+ 
\frac{\lbrack W_n,  W_0 \rbrack}{(1+b^{-n})(1-b^{-1})},
\label{eq:nsolvG}
\\
\tilde G_n &= G_n + \frac{\lbrack W_0, W_n\rbrack }{1-b^{-1}},
\label{eq:nsolvGt}
\\
W_{-n} &= \frac{\lbrack W_0, G_n \rbrack_b}{b-1}, 
\label{eq:nsolvWm}
\\
W_{n+1} &= \frac{\lbrack G_n, W_1\rbrack_b}{b-1}. 
\label{eq:nsolvWp}
\end{align}
\end{lemma}

\begin{lemma}
\label{lem:nsigSact}
{\rm (See \cite[Lemma~5.3]{alternating}.)}
Consider the maps $\sigma$, $\dagger$  from
Lemma 
\ref{lem:nAAut}.
 For $k \in \mathbb N$,
\begin{enumerate}
\item[\rm(i)]
the map $\sigma $ sends
\begin{align*}
W_{-k} \mapsto W_{k+1}, \qquad \quad
W_{k+1} \mapsto W_{-k}, \qquad \quad
G_{k} \mapsto \tilde G_{k}, \qquad \quad
\tilde G_{k} \mapsto G_{k};
\end{align*}
\item[\rm (ii)] the map $\dagger$ sends
\begin{align*}
W_{-k} \mapsto W_{-k}, \qquad \quad
W_{k+1} \mapsto W_{k+1}, \qquad \quad
G_{k} \mapsto \tilde G_{k}, \qquad \quad
\tilde G_{k} \mapsto G_{k}.
\end{align*}
\end{enumerate}
\end{lemma}
\begin{lemma}
\label{lem:ms2}
Pick nonzero $\lambda_0, \lambda_1\in \mathbb F$ and consider the corresponding automorphism of $U^+_q$ from Lemma
\ref{lem:ms}.  For $k \in \mathbb N$ the automorphism sends
\begin{align*}
W_{-k} \mapsto \lambda^{k+1}_0 \lambda^k_1 W_{-k}, 
\qquad 
W_{k+1} \mapsto \lambda^{k}_0 \lambda^{k+1}_1 W_{k+1}, 
\qquad 
G_k \mapsto \lambda^{k}_0 \lambda^k_1 G_{k}, 
\qquad 
{\tilde G}_k \mapsto \lambda^{k}_0 \lambda^k_1 {\tilde G}_{k}. 
\end{align*}
\end{lemma}
\begin{proof} Use induction on the ordering of the alternating elements given in Lemma \ref{lem:nrecgen}.
\end{proof}

\noindent We mention some facts for later use.


\begin{proposition} \label{lem:fourP} For $k \in \mathbb N$ the following {\rm (i)--(iv)} hold in $U^+_q$.
\begin{enumerate}
\item[\rm (i)] 
$W_{-k}$ satisfies
\begin{align*}
\lbrack W_0, W_{-k}\rbrack=0, \qquad \qquad 
\lbrack W_1, \lbrack W_1, \lbrack W_1, W_{-k}\rbrack_b \rbrack_{b^{-1}} \rbrack &= 0.
\end{align*}
\item[\rm (ii)] 
$W_{k+1}$ satisfies
\begin{align*}
\lbrack W_0, \lbrack W_0, \lbrack W_0, W_{k+1}\rbrack_b \rbrack_{b^{-1}} \rbrack = 0,
\qquad \qquad \lbrack W_1, W_{k+1}\rbrack =0. 
\end{align*}
\item[\rm (iii)] $G_{k+1} $ satisfies
\begin{align*} 
\lbrack W_0, \lbrack W_0, G_{k+1}\rbrack_b \rbrack=0, \qquad \qquad
\lbrack \lbrack G_{k+1}, W_1\rbrack_b,  W_1\rbrack = 0.  
\end{align*}
\item[\rm (iv)] ${\tilde G}_{k+1} $ satisfies
\begin{align*}
\lbrack \lbrack {\tilde G}_{k+1}, W_0\rbrack_b,  W_0\rbrack = 0,   
\qquad \qquad \lbrack W_1, \lbrack W_1, {\tilde G}_{k+1}\rbrack_b \rbrack =0.
\end{align*}
\end{enumerate}
\end{proposition}
\begin{proof} (i) The equation $\lbrack W_0, W_{-k} \rbrack=0$ is from Lemma \ref{lem:nrel2}.
Observe that
\begin{align*}
\lbrack W_1, &\lbrack W_1, \lbrack W_1, W_{-k}\rbrack_b \rbrack_{b^{-1}} \rbrack 
\\
&=
\lbrack W_1, \lbrack W_1, \lbrack W_1, W_{-k}\rbrack \rbrack_{b} \rbrack_{b^{-1}} \\
&=(1-b^{-1}) \lbrack W_1, \lbrack W_1, G_{k+1} - \tilde G_{k+1} \rbrack_{b} \rbrack_{b^{-1}} \\
&=(1-b^{-1}) \lbrack W_1, \lbrack W_1, G_{k+1} \rbrack_{b} \rbrack_{b^{-1}} -(1-b^{-1}) \lbrack W_1, \lbrack W_1, \tilde G_{k+1} \rbrack_{b} \rbrack_{b^{-1}} \\
&=(1-b^{-1}) \lbrack W_1, \lbrack W_1, G_{k+1} \rbrack_{b^{-1}} \rbrack_{b} +(b^{-1}-1) \lbrack W_1, \lbrack W_1, \tilde G_{k+1} \rbrack_{b} \rbrack_{b^{-1}} \\
&=(b^{-1}-1) b^{-1}\lbrack W_1, \lbrack  G_{k+1}, W_1 \rbrack_b \rbrack_b +(b^{-1}-1) \lbrack W_1, \lbrack  W_1, \tilde G_{k+1} \rbrack_b \rbrack_{b^{-1}} \\
&=(b^{-1}-1)(1-b^{-1}) \lbrack W_1, W_{k+2}\rbrack_b +(b^{-1}-1)(b-1) \lbrack W_1, W_{k+2} \rbrack_{b^{-1}} \\
&=(b^{-1}-1)(b-b^{-1}) \lbrack W_1, W_{k+2}\rbrack \\
&=0.
\end{align*}
\noindent (ii) Apply the automorphism $\sigma$ to each side of the equations in (i) above.
\\
\noindent (iii) We have
\begin{align*}
\lbrack W_0, \lbrack W_0, G_{k+1}\rbrack_b \rbrack &=(b-1) \lbrack W_0,  W_{-k-1} \rbrack =0.
\end{align*}
\noindent We also have
\begin{align*}
\lbrack \lbrack G_{k+1}, W_1\rbrack_b,  W_1\rbrack &= (b-1) \lbrack W_{k+2}, W_1\rbrack = 0.
\end{align*}
\noindent (iv) Apply the automorphism $\sigma$ to each side of the equations in (iii) above.
\end{proof}

\section{TD systems of $q$-Serre type}

We return our attention to the nontrivial TD system $\Phi=(A;\lbrace E_i\rbrace_{i=0}^d;A^*;\lbrace E^*_i\rbrace_{i=0}^d)$ on $V$. Recall that $\Phi$ has eigenvalue sequence
$\lbrace \theta_i \rbrace_{i=0}^d$ and dual eigenvalue sequence $\lbrace \theta^*_i \rbrace_{i=0}^d$.
\begin{definition} \label{def:qGeom} \rm (See \cite[Definition~2.6]{nonnil}.)  The TD system $\Phi$ is said to have
 {\it $q$-Serre type} whenever $\theta_i = b \theta_{i-1}$ and $\theta^*_i = b^{-1} \theta^*_{i-1}$ for $1 \leq i \leq d$.
\end{definition}
\noindent We refer the reader to \cite{CurtH, hasan2, hasan, shape, TD00, qSerre, tdanduqsl2hat, nonnil, qtet, evalTetq, aayush} for background information on TD systems of $q$-Serre type.
\medskip

\noindent From now until the end of Section 10, we assume that $\Phi$ has $q$-Serre type.
\begin{lemma} \label{lem:qS} The elements $A$, $A^*$ satisfy the $q$-Serre relations
\begin{align}
\lbrack A, \lbrack A, \lbrack A, A^*\rbrack_b \rbrack_{b^{-1}} \rbrack=0,
\qquad \qquad 
\lbrack A^*, \lbrack A^*, \lbrack A^*, A\rbrack_b \rbrack_{b^{-1}}
\rbrack=0.
\label{eq:qSSS}
\end{align}
\end{lemma} 
\begin{proof} We invoke Lemma \ref{eq:lastchancedolangradyS992}.
The scalars
\begin{align}
\label{eq:vvalues}
\beta = b+b^{-1}, \qquad \gamma=0, \qquad \gamma^*=0, \qquad \varrho=0, \qquad \varrho^*=0
\end{align}
satisfy the conditions (i)--(v) in Lemma \ref{eq:lastchancedolangradyS992}. Using \eqref{eq:vvalues} the tridiagonal relations \eqref{eq:TD1}, \eqref{eq:TD2} become the relations \eqref{eq:qSSS}.
\end{proof}
\begin{corollary}
\label{cor:modU} The vector space $V$ becomes
a $U^+_q$-module on which $W_0=A$ and $W_1=A^*$. The $U^+_q$-module $V$ is irreducible.
\end{corollary}
\begin{proof} The first assertion follows from Definition \ref{def:nposp} and
Lemma \ref{lem:qS}.
The last assertion follows from condition (iv) in Definition \ref{def:tdp}.
\end{proof}

 \noindent Next we consider how the alternating elements of $U^+_q$ act on the six decompositions of $V$ from Example \ref{thm:sixdecp}.
We will use the following result.
\begin{lemma} \label{lem:coeff} For $0 \leq i,j\leq d$,
\begin{align*}
b \theta_i = \theta_j  \quad &{\mbox{iff}} \quad i=j-1 \quad {\mbox{iff}} \quad \theta^*_i =  b\theta^*_j;\\
\theta_i =  \theta_j  \quad &{\mbox{iff}} \qquad i=j \quad\;\; {\mbox{iff}} \quad \theta^*_i = \theta^*_j;\\
 \theta_i =b\theta_j  \quad &{\mbox{iff} }\quad i=j+1 \quad {\mbox{iff}} \quad b \theta^*_i =  \theta^*_j.
\end{align*}
\end{lemma}
\begin{proof} By Definition \ref{def:qGeom} and since $b$ is not a root of unity.
\end{proof}

\begin{lemma} \label{lem:EWE} For  $k \in \mathbb N$ and $0 \leq i,j \leq d$ the following {\rm (i)--(iv)} hold on the $U^+_q$-module $V$.
\begin{enumerate}
\item[\rm (i)] 
$W_{-k}$ satisfies
\begin{align*}
E_i W_{-k} E_j = 0 \quad &{\mbox{\rm if $i\not=j$}},\\
E^*_i W_{-k} E^*_j = 0 \quad &{\mbox{\rm if $\vert i-j\vert > 1$}}.
\end{align*}
\item[\rm (ii)] 
$W_{k+1}$ satisfies
\begin{align*}
E_i W_{k+1} E_j = 0 \quad &{\mbox{\rm if $\vert i-j\vert >1$}},\\
E^*_i W_{k+1} E^*_j = 0 \quad &{\mbox{\rm if $i\not=j$}}.
\end{align*}
\item[\rm (iii)] 
$G_{k+1}$ satisfies
\begin{align*}
E_i G_{k+1} E_j = 0 \quad &{\mbox{\rm if $j-i \not\in \lbrace 0,1\rbrace $}},\\
E^*_i G_{k+1} E^*_j = 0 \quad &{\mbox{\rm if $j-i \not\in \lbrace 0,1\rbrace $}}.
\end{align*}
\item[\rm (iv)] ${\tilde G}_{k+1} $ satisfies
\begin{align*}
E_i \tilde G_{k+1} E_j &= 0 \quad {\mbox{\rm if $i-j \not\in \lbrace 0,1\rbrace $}},\\
E^*_i \tilde G_{k+1} E^*_j &= 0 \quad {\mbox{\rm if $i-j \not\in \lbrace 0,1\rbrace $}}.
\end{align*}
\end{enumerate}
\end{lemma}
\begin{proof} For each equation in Proposition \ref{lem:fourP} that involves $W_0$, multiply each term on the left  by $E_i$ and on the right by $E_j$. Simplify the
result using $E_iA = \theta_i E_i$ and $AE_j = \theta_j E_j$ along with $A=W_0$.
For each equation in Proposition  \ref{lem:fourP} that involves $W_1$, multiply each term on the left by $E^*_i$ and on the right by $E^*_j$. Simplify the
result using $E^*_i A^* = \theta^*_i E^*_i$ and $A^*E^*_j = \theta^*_j E^*_j$ along with $A^*=W_1$. By these comments, the following equations hold on $V$:
\begin{align*} 
E_i W_{-k} E_j (\theta_i - \theta_j)&=0,\\
E^*_i W_{-k} E^*_j (\theta^*_i - \theta^*_j)(b \theta^*_i -  \theta^*_j)(b^{-1}\theta^*_i - \theta^*_j)&=0, \\
E_i W_{k+1} E_j (\theta_i - \theta_j)(b \theta_i -  \theta_j)(b^{-1} \theta_i - \theta_j)&=0, \\
E^*_i W_{k+1} E^*_j (\theta^*_i - \theta^*_j)&=0,\\
E_i G_{k+1} E_j (\theta_i - \theta_j)(b\theta_i -  \theta_j)&=0,\\
E^*_i G_{k+1} E^*_j (\theta^*_i - \theta^*_j)(\theta^*_i- b\theta^*_j)&=0, \\
E_i \tilde G_{k+1} E_j (\theta_i - \theta_j)( \theta_i -  b\theta_j)&=0, \\
E^*_i \tilde  G_{k+1} E^*_j (\theta^*_i - \theta^*_j)(b \theta^*_i -  \theta^*_j)&=0.
\end{align*}
Evaluate the above equations using Lemma \ref{lem:coeff} to get the result.
\end{proof}
\noindent The following result is a reformulation of Lemma \ref{lem:EWE}.

\begin{lemma} \label{lem:nv} We refer to the $U^+_q$-module $V$ from Corollary \ref{cor:modU}. For  $k \in \mathbb N$ and $0 \leq j \leq d$ the following {\rm (i)--(iv)} hold.
\begin{enumerate}
\item[\rm (i)] 
$W_{-k}$ satisfies
\begin{align*}
W_{-k} E_j V& \subseteq E_jV, \\
W_{-k} E^*_j V & \subseteq E^*_{j-1}V+E^*_jV+ E^*_{j+1}V.
\end{align*}
\item[\rm (ii)] 
$W_{k+1}$ satisfies
\begin{align*}
W_{k+1} E_j V & \subseteq E_{j-1}V+E_jV+ E_{j+1}V, \\
W_{k+1} E^*_j V & \subseteq E^*_jV.
\end{align*}
\item[\rm (iii)] 
$G_{k+1}$ satisfies
\begin{align*}
G_{k+1} E_jV &\subseteq E_{j-1}V + E_jV,\\
G_{k+1} E^*_jV & \subseteq E^*_{j-1}V+E^*_jV.
\end{align*}
\item[\rm (iv)] ${\tilde G}_{k+1} $ satisfies
\begin{align*}
 \tilde G_{k+1} E_jV &\subseteq E_jV+E_{j+1}V,\\
  \tilde G_{k+1} E^*_jV&\subseteq E^*_jV+E^*_{j+1}V.
\end{align*}
\end{enumerate}
\noindent We are using the convention
\begin{align*} 
 E_{-1}=0, \qquad E_{d+1}=0, \qquad E^*_{-1}=0, \qquad E^*_{d+1}=0.
 \end{align*}
\end{lemma}
\begin{proof} (i) Using Lemma \ref{lem:EWE}(i) we obtain
\begin{align*}  
      W_{-k} E_jV &= I W_{-k} E_jV 
                         = \sum_{i=0}^d E_i W_{-k} E_j V 
                           = E_j W_{-k} E_j V 
                           \subseteq E_jV, 
                           \end{align*}
                           and also
                           \begin{align*}
 W_{-k} E^*_jV &= I W_{-k} E^*_jV 
                         = \sum_{i=0}^d E^*_i W_{-k} E^*_j V 
                           = \sum_{i=j-1}^{j+1} E^*_i W_{-k} E^*_j V 
                           \subseteq E^*_{j-1} V + E^*_jV+E^*_{j+1}V.
\end{align*}
\noindent (ii)--(iv) Similar to the proof of (i) above.
\end{proof}

\begin{proposition} \label{lem:flagAct} 
We refer to the $U^+_q$-module $V$ from Corollary \ref{cor:modU}.
In the table below, we show how the alternating elements of $U^+_q$ act on the flags from Definition  \ref{lem:fourflag}. 
For $k \in \mathbb N$,
\bigskip

\centerline{
\begin{tabular}[t]{c|cccc}
      {\rm flag name} &{\rm $W_{-k}$ action } & {\rm $W_{k+1}$ action} &{\rm $G_{k+1}$ action} & {\rm ${\tilde G}_{k+1}$ action}
      \\ \hline  \hline
	 $\lbrack 0\rbrack $ &
	{\rm stabilize }  & \rm raise & \rm stabilize   & \rm raise
\\	
	$\lbrack D\rbrack $ & \rm stabilize & \rm raise & \rm raise & \rm stabilize
	 \\
       $\lbrack 0^*\rbrack $ & \rm raise &\rm stabilize &\rm stabilize & \rm raise
          \\ 
	$\lbrack D^*\rbrack $ & \rm raise & \rm stabilize & \rm raise & \rm stabilize
		\end{tabular}}
\medskip

\end{proposition}
\begin{proof}  Use Lemmas \ref{lem:RR}--\ref{lem:SS} along with Lemma \ref{lem:nv}.
\end{proof}

\noindent We now give our first main result.
\begin{theorem}
\label{thm:WWaction} We refer to the $U^+_q$-module $V$ from Corollary \ref{cor:modU}.
Let $\lbrace \mathcal V_i \rbrace_{i=0}^d$ 
denote a decomposition
of $V$ from Example \ref{thm:sixdecp}.
Then for $k \in \mathbb N$ and $0 \leq i \leq d$
the actions of $W_{-k}$ and $W_{k+1}$ on $\mathcal V_i$ are described in the table below.

\bigskip

\centerline{
\begin{tabular}[t]{c|c|c}
      {\rm decomp. name} &{\rm action of $W_{-k}$ on $\mathcal V_i$} & {\rm action of $W_{k+1}$ on $\mathcal V_i$}
      \\ \hline  \hline
	$\lbrack 0,D\rbrack $ &
	$ W_{-k} \mathcal V_i\subseteq \mathcal V_i$  & 
	$W_{k+1} \mathcal V_i \subseteq \mathcal V_{i-1}+ \mathcal V_i+\mathcal V_{i+1}$ 
\\	
	$\lbrack 0^*,D^*\rbrack $ &
	$W_{-k} \mathcal V_i \subseteq \mathcal V_{i-1}+ \mathcal V_i+\mathcal V_{i+1}$  &
	$ W_{k+1}\mathcal V_i\subseteq \mathcal V_i$   \\
      $\lbrack 0^*,0\rbrack $ & 
	$W_{-k} \mathcal V_i\subseteq \mathcal V_i + \mathcal V_{i+1}$ &
	$W_{k+1} \mathcal V_i \subseteq \mathcal V_{i-1}+\mathcal V_i$   \\ 
       $\lbrack 0^*,D\rbrack $ &
        $W_{-k} \mathcal V_i\subseteq \mathcal V_i+\mathcal V_{i+1}$ &
	$W_{k+1} \mathcal V_i \subseteq \mathcal V_{i-1}+\mathcal V_i$   \\ 
       $\lbrack D^*,0\rbrack $ & 
	$W_{-k} \mathcal V_i\subseteq \mathcal V_i +  \mathcal V_{i+1}$ &
	$W_{k+1} \mathcal V_i \subseteq \mathcal V_{i-1}+\mathcal V_i$   \\ 
	$\lbrack D^*,D\rbrack $ &
	$W_{-k} \mathcal V_i\subseteq \mathcal V_i + \mathcal V_{i+1}$ &
	$W_{k+1} \mathcal V_i \subseteq \mathcal V_{i-1}+\mathcal V_i$  
	\end{tabular}}
\medskip

\end{theorem}
\begin{proof} For the decompositions $\lbrack 0,D\rbrack$ and $\lbrack 0^*, D^*\rbrack$ use
Lemma \ref{lem:nv}(i),(ii). 
For the remaining four decompositions, use Lemmas \ref{lem:RS}, \ref{lem:SR} and Proposition \ref{lem:flagAct}.
\end{proof}
\noindent Next we use the tetrahedron diagram to illustrate Theorem \ref{thm:WWaction}. 
\medskip
\noindent
Pick $k \in \mathbb N$. The following picture shows how $W_{-k}$ acts on the  decompositions of $V$ from the tetrahedron diagram, for $d=8$:

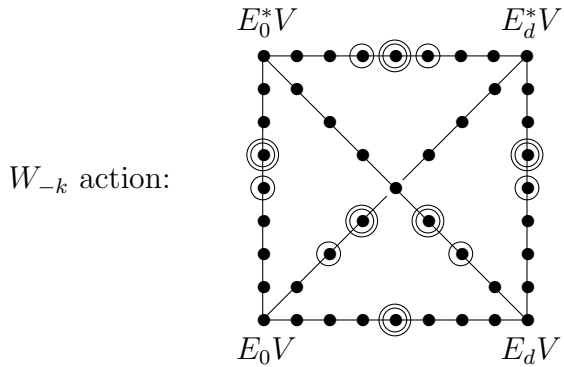
\begin{center}
\begin{picture}(100,80)

\put(-50,50){\line(1,0){50}}
\put(0,50){\line(1,0){50}}

\put(50,50){\line(0,-1){50}}
\put(50,0){\line(0,-1){50}}

\put(50,-50){\line(-1,0){50}}
\put(0,-50){\line(-1,0){50}}

\put(-50,-50){\line(0,1){50}}
\put(-50,0){\line(0,1){50}}

\put(0,0){\line(-1,1){25}}
\put(-25,25){\line(-1,1){25}}
\put(0,0){\line(1,-1){25}}
\put(25,-25){\line(1,-1){25}}

\put(3,3){\line(1,1){22}}
\put(25,25){\line(1,1){25}}
\put(-3,-3){\line(-1,-1){22}}
\put(-25,-25){\line(-1,-1){25}}

\put(-60,60){$E^*_0V$}
\put(40,60){$E^*_dV$}
\put(-60,-65){$E_0V$}
\put(40,-65){$E_dV$}
\put(-150,0){ $W_{-k}$ action:}
\put(-52.5,47){$\bullet$}
\put(-40,47){$\bullet$}
\put(-27.5,47){$\bullet$}
\put(-15,47){$\bullet$}
\put(-2.5,47){$\bullet$}

\put(47,47){$\bullet$}
\put(34.5,47){$\bullet$}
\put(22,47){$\bullet$}
\put(9.5,47){$\bullet$}

\put(-12.5,50){\circle{9}}
\put(0,50){\circle{9}}
\put(0,50){\circle{12}}
\put(12.5,50){\circle{9}}
\put(-52.5,-53){$\bullet$}
\put(-40,-53){$\bullet$}
\put(-27.5,-53){$\bullet$}
\put(-15,-53){$\bullet$}
\put(-2.5,-53){$\bullet$}

\put(47.5,-53){$\bullet$}
\put(35,-53){$\bullet$}
\put(22.5,-53){$\bullet$}
\put(10,-53){$\bullet$}

\put(0,-50){\circle{9}}
\put(0,-50){\circle{12}}
\put(-52.5,34.5){$\bullet$}
\put(-52.5,22){$\bullet$}
\put(-52.5,9.5){$\bullet$}
\put(-52.5,-3){$\bullet$}
\put(-52.5,-15.5){$\bullet$}
\put(-52.5,-28){$\bullet$}
\put(-52.5,-40.5){$\bullet$}

\put(-50,12.5){\circle{9}}
\put(-50,12.5){\circle{12}}
\put(-50,0){\circle{9}}
\put(50,12.5){\circle{9}}
\put(50,12.5){\circle{12}}
\put(50,0){\circle{9}}

\put(-25,-25){\circle{9}}
\put(-12.5,-12.5){\circle{12}}
\put(-12.5,-12.5){\circle{9}}
\put(25,-25){\circle{9}}
\put(12.5,-12.5){\circle{12}}
\put(12.5,-12.5){\circle{9}}

\put(47.5,34.5){$\bullet$}
\put(47.5,22){$\bullet$}
\put(47.5,9.5){$\bullet$}
\put(47.5,-3){$\bullet$}
\put(47.5,-15.5){$\bullet$}
\put(47.5,-28){$\bullet$}
\put(47.5,-40.5){$\bullet$}

\put(-40,34.5){$\bullet$}
\put(-27.5,22){$\bullet$}
\put(-15,9.5){$\bullet$}
\put(-2.5,-3){$\bullet$}

\put(35,-40.5){$\bullet$}
\put(22.5,-28){$\bullet$}
\put(10,-15.5){$\bullet$}

\put(35,34.5){$\bullet$}
\put(22.5,22){$\bullet$}
\put(10,9.5){$\bullet$}
\put(-15,-15.5){$\bullet$}
\put(-27.5,-28){$\bullet$}
\put(-40,-40.5){$\bullet$}

\end{picture}
\end{center}

\vspace{2.5cm}

\noindent The following picture shows how $W_{k+1}$ acts on the  decompositions of $V$ from the tetrahedron diagram,  for $d=8$:
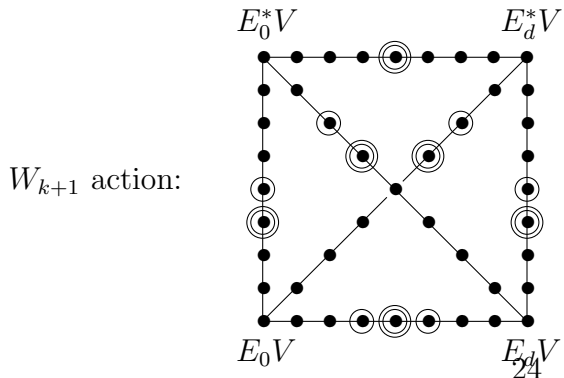
\begin{center}

\begin{picture}(100,80)

\put(-50,50){\line(1,0){50}}
\put(0,50){\line(1,0){50}}

\put(50,50){\line(0,-1){50}}
\put(50,0){\line(0,-1){50}}

\put(50,-50){\line(-1,0){50}}
\put(0,-50){\line(-1,0){50}}

\put(-50,-50){\line(0,1){50}}
\put(-50,0){\line(0,1){50}}

\put(0,0){\line(-1,1){25}}
\put(-25,25){\line(-1,1){25}}
\put(0,0){\line(1,-1){25}}
\put(25,-25){\line(1,-1){25}}

\put(3,3){\line(1,1){22}}
\put(25,25){\line(1,1){25}}
\put(-3,-3){\line(-1,-1){22}}
\put(-25,-25){\line(-1,-1){25}}

\put(-60,60){$E^*_0V$}
\put(40,60){$E^*_dV$}
\put(-60,-65){$E_0V$}
\put(40,-65){$E_dV$}
\put(-150,0){ $W_{k+1}$ action:}
\put(-52.5,47){$\bullet$}
\put(-40,47){$\bullet$}
\put(-27.5,47){$\bullet$}
\put(-15,47){$\bullet$}
\put(-2.5,47){$\bullet$}

\put(47,47){$\bullet$}
\put(34.5,47){$\bullet$}
\put(22,47){$\bullet$}
\put(9.5,47){$\bullet$}

\put(-12.5,-50){\circle{9}}
\put(0,-50){\circle{9}}
\put(0,-50){\circle{12}}
\put(12.5,-50){\circle{9}}
\put(-52.5,-53){$\bullet$}
\put(-40,-53){$\bullet$}
\put(-27.5,-53){$\bullet$}
\put(-15,-53){$\bullet$}
\put(-2.5,-53){$\bullet$}

\put(47.5,-53){$\bullet$}
\put(35,-53){$\bullet$}
\put(22.5,-53){$\bullet$}
\put(10,-53){$\bullet$}

\put(0,50){\circle{9}}
\put(0,50){\circle{12}}
\put(-52.5,34.5){$\bullet$}
\put(-52.5,22){$\bullet$}
\put(-52.5,9.5){$\bullet$}
\put(-52.5,-3){$\bullet$}
\put(-52.5,-15.5){$\bullet$}
\put(-52.5,-28){$\bullet$}
\put(-52.5,-40.5){$\bullet$}

\put(-50,-12.5){\circle{9}}
\put(-50,-12.5){\circle{12}}
\put(-50,0){\circle{9}}
\put(50,-12.5){\circle{9}}
\put(50,-12.5){\circle{12}}
\put(50,0){\circle{9}}

\put(-25,25){\circle{9}}
\put(-12.5,12.5){\circle{12}}
\put(-12.5,12.5){\circle{9}}
\put(25,25){\circle{9}}
\put(12.5,12.5){\circle{12}}
\put(12.5,12.5){\circle{9}}

\put(47.5,34.5){$\bullet$}
\put(47.5,22){$\bullet$}
\put(47.5,9.5){$\bullet$}
\put(47.5,-3){$\bullet$}
\put(47.5,-15.5){$\bullet$}
\put(47.5,-28){$\bullet$}
\put(47.5,-40.5){$\bullet$}

\put(-40,34.5){$\bullet$}
\put(-27.5,22){$\bullet$}
\put(-15,9.5){$\bullet$}
\put(-2.5,-3){$\bullet$}

\put(35,-40.5){$\bullet$}
\put(22.5,-28){$\bullet$}
\put(10,-15.5){$\bullet$}

\put(35,34.5){$\bullet$}
\put(22.5,22){$\bullet$}
\put(10,9.5){$\bullet$}
\put(-15,-15.5){$\bullet$}
\put(-27.5,-28){$\bullet$}
\put(-40,-40.5){$\bullet$}

\end{picture}
\end{center}
\vspace{2.5cm}

\noindent We now give our second main result.
\begin{theorem}
\label{thm:GGaction}  We refer to the $U^+_q$-module $V$ from Corollary \ref{cor:modU}.
Let $\lbrace \mathcal V_i \rbrace_{i=0}^d$ 
denote a decomposition
of $V$ from  Example \ref{thm:sixdecp}.
Then for $k \in \mathbb N$ and  $0 \leq i \leq d$
the actions of $G_{k+1} $ and $\tilde G_{k+1}$ on $\mathcal V_i$ are described in the table below.

\bigskip

\centerline{
\begin{tabular}[t]{c|c|c}
      {\rm decomp. name} &{\rm action of $G_{k+1}$ on $\mathcal V_i$} & {\rm action of $\tilde G_{k+1}$ on $\mathcal V_i$}
      \\ \hline  \hline
	$\lbrack 0,D\rbrack $ &
	$ G_{k+1} \mathcal V_i\subseteq \mathcal V_{i-1} +\mathcal V_i$  & 
	$\tilde G_{k+1} \mathcal V_i \subseteq \mathcal V_i+\mathcal V_{i+1}$ 
\\	
	$\lbrack 0^*,D^*\rbrack $ &
	$ G_{k+1} \mathcal V_i \subseteq \mathcal V_{i-1}+\mathcal V_i$  &
	$ \tilde G_{k+1}\mathcal V_i\subseteq \mathcal V_i + \mathcal V_{i+1} $   \\
       $\lbrack 0^*,0\rbrack $ & 
	$G_{k+1} \mathcal V_i\subseteq \mathcal V_i$ &
	$\tilde G_{k+1} \mathcal V_i \subseteq \mathcal V_{i-1}+ \mathcal V_i + \mathcal V_{i+1} $   \\ 
       $\lbrack 0^*,D\rbrack $ &
        $G_{k+1} \mathcal V_i\subseteq \mathcal V_{i-1}+\mathcal V_i$ &
	$\tilde G_{k+1} \mathcal V_i \subseteq \mathcal V_i + \mathcal V_{i+1}$   \\ 
	       $\lbrack D^*,0\rbrack $ & 
	$G_{k+1} \mathcal V_i\subseteq \mathcal V_i + \mathcal V_{i+1}$ &
	$\tilde G_{k+1} \mathcal V_i \subseteq \mathcal V_{i-1}+\mathcal V_i$   \\ 
	$\lbrack D^*,D\rbrack $ &
	$G_{k+1} \mathcal V_i\subseteq \mathcal V_{i-1} + \mathcal V_i + \mathcal V_{i+1}$ &
	$\tilde G_{k+1} \mathcal V_i \subseteq \mathcal V_i$  
	\end{tabular}}
\medskip

\end{theorem}
\begin{proof}
For the decompositions $\lbrack 0,D\rbrack$ and $\lbrack 0^*, D^*\rbrack$ use
Lemma \ref{lem:nv}(iii),(iv). For the remaining four decompositions,
use Lemmas \ref{lem:RR}--\ref{lem:SS} and Proposition \ref{lem:flagAct}.
\end{proof}

\noindent Next we use the tetrahedron diagram to illustrate Theorem \ref{thm:GGaction}. 
\medskip

\noindent
Pick $k \in \mathbb N$. The following picture shows how $G_{k+1}$ acts on the  decompositions of $V$ from the tetrahedron diagram, for $d=8$:

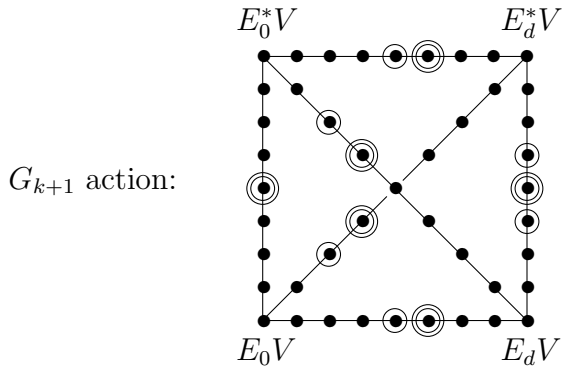
\begin{center}
\begin{picture}(100,80)

\put(-50,50){\line(1,0){50}}
\put(0,50){\line(1,0){50}}

\put(50,50){\line(0,-1){50}}
\put(50,0){\line(0,-1){50}}

\put(50,-50){\line(-1,0){50}}
\put(0,-50){\line(-1,0){50}}

\put(-50,-50){\line(0,1){50}}
\put(-50,0){\line(0,1){50}}

\put(0,0){\line(-1,1){25}}
\put(-25,25){\line(-1,1){25}}
\put(0,0){\line(1,-1){25}}
\put(25,-25){\line(1,-1){25}}

\put(3,3){\line(1,1){22}}
\put(25,25){\line(1,1){25}}
\put(-3,-3){\line(-1,-1){22}}
\put(-25,-25){\line(-1,-1){25}}

\put(-60,60){$E^*_0V$}
\put(40,60){$E^*_dV$}
\put(-60,-65){$E_0V$}
\put(40,-65){$E_dV$}
\put(-150,0){ $G_{k+1}$ action:}
\put(-52.5,47){$\bullet$}
\put(-40,47){$\bullet$}
\put(-27.5,47){$\bullet$}
\put(-15,47){$\bullet$}
\put(-2.5,47){$\bullet$}

\put(47,47){$\bullet$}
\put(34.5,47){$\bullet$}
\put(22,47){$\bullet$}
\put(9.5,47){$\bullet$}

\put(0,50){\circle{9}}
\put(12.5,50){\circle{12}}
\put(12.5,50){\circle{9}}
\put(-52.5,-53){$\bullet$}
\put(-40,-53){$\bullet$}
\put(-27.5,-53){$\bullet$}
\put(-15,-53){$\bullet$}
\put(-2.5,-53){$\bullet$}

\put(47.5,-53){$\bullet$}
\put(35,-53){$\bullet$}
\put(22.5,-53){$\bullet$}
\put(10,-53){$\bullet$}

\put(0,-50){\circle{9}}
\put(12.5,-50){\circle{12}}
\put(12.5,-50){\circle{9}}
\put(-52.5,34.5){$\bullet$}
\put(-52.5,22){$\bullet$}
\put(-52.5,9.5){$\bullet$}
\put(-52.5,-3){$\bullet$}
\put(-52.5,-15.5){$\bullet$}
\put(-52.5,-28){$\bullet$}
\put(-52.5,-40.5){$\bullet$}

\put(-50,0){\circle{12}}
\put(-50,0){\circle{9}}
\put(50,12.5){\circle{9}}
\put(50,0){\circle{12}}
\put(50,0){\circle{9}}
\put(50,-12.5){\circle{9}}

\put(-25,-25){\circle{9}}
\put(-12.5,-12.5){\circle{12}}
\put(-12.5,-12.5){\circle{9}}
\put(-25,25){\circle{9}}
\put(-12.5,12.5){\circle{12}}
\put(-12.5, 12.5){\circle{9}}

\put(47.5,34.5){$\bullet$}
\put(47.5,22){$\bullet$}
\put(47.5,9.5){$\bullet$}
\put(47.5,-3){$\bullet$}
\put(47.5,-15.5){$\bullet$}
\put(47.5,-28){$\bullet$}
\put(47.5,-40.5){$\bullet$}

\put(-40,34.5){$\bullet$}
\put(-27.5,22){$\bullet$}
\put(-15,9.5){$\bullet$}
\put(-2.5,-3){$\bullet$}

\put(35,-40.5){$\bullet$}
\put(22.5,-28){$\bullet$}
\put(10,-15.5){$\bullet$}

\put(35,34.5){$\bullet$}
\put(22.5,22){$\bullet$}
\put(10,9.5){$\bullet$}
\put(-15,-15.5){$\bullet$}
\put(-27.5,-28){$\bullet$}
\put(-40,-40.5){$\bullet$}

\end{picture}
\end{center}

\vspace{2.5cm}

\noindent The following picture shows how ${\tilde G}_{k+1}$ acts on the  decompositions of $V$ from the tetrahedron diagram,  for $d=8$:
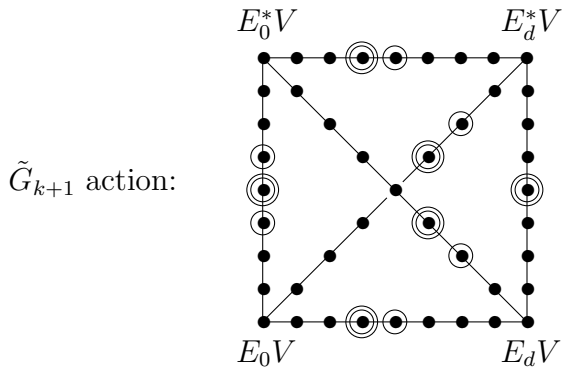
\begin{center}

\begin{picture}(100,80)

\put(-50,50){\line(1,0){50}}
\put(0,50){\line(1,0){50}}

\put(50,50){\line(0,-1){50}}
\put(50,0){\line(0,-1){50}}

\put(50,-50){\line(-1,0){50}}
\put(0,-50){\line(-1,0){50}}

\put(-50,-50){\line(0,1){50}}
\put(-50,0){\line(0,1){50}}

\put(0,0){\line(-1,1){25}}
\put(-25,25){\line(-1,1){25}}
\put(0,0){\line(1,-1){25}}
\put(25,-25){\line(1,-1){25}}

\put(3,3){\line(1,1){22}}
\put(25,25){\line(1,1){25}}
\put(-3,-3){\line(-1,-1){22}}
\put(-25,-25){\line(-1,-1){25}}

\put(-60,60){$E^*_0V$}
\put(40,60){$E^*_dV$}
\put(-60,-65){$E_0V$}
\put(40,-65){$E_dV$}
\put(-150,0){ ${\tilde G}_{k+1}$ action:}
\put(-52.5,47){$\bullet$}
\put(-40,47){$\bullet$}
\put(-27.5,47){$\bullet$}
\put(-15,47){$\bullet$}
\put(-2.5,47){$\bullet$}

\put(47,47){$\bullet$}
\put(34.5,47){$\bullet$}
\put(22,47){$\bullet$}
\put(9.5,47){$\bullet$}

\put(-12.5,-50){\circle{9}}
\put(0,-50){\circle{9}}
\put(-12.5,-50){\circle{12}}
\put(-52.5,-53){$\bullet$}
\put(-40,-53){$\bullet$}
\put(-27.5,-53){$\bullet$}
\put(-15,-53){$\bullet$}
\put(-2.5,-53){$\bullet$}

\put(47.5,-53){$\bullet$}
\put(35,-53){$\bullet$}
\put(22.5,-53){$\bullet$}
\put(10,-53){$\bullet$}

\put(0,50){\circle{9}}
\put(-12.5,50){\circle{12}}
\put(-12.5,50){\circle{9}}
\put(-52.5,34.5){$\bullet$}
\put(-52.5,22){$\bullet$}
\put(-52.5,9.5){$\bullet$}
\put(-52.5,-3){$\bullet$}
\put(-52.5,-15.5){$\bullet$}
\put(-52.5,-28){$\bullet$}
\put(-52.5,-40.5){$\bullet$}

\put(-50,12.5){\circle{9}}
\put(-50,0){\circle{12}}
\put(-50,0){\circle{9}}
\put(-50,-12.5){\circle{9}}
\put(50,0){\circle{12}}
\put(50,0){\circle{9}}

\put(25,-25){\circle{9}}
\put(12.5,-12.5){\circle{12}}
\put(12.5,-12.5){\circle{9}}
\put(25,25){\circle{9}}
\put(12.5,12.5){\circle{12}}
\put(12.5,12.5){\circle{9}}

\put(47.5,34.5){$\bullet$}
\put(47.5,22){$\bullet$}
\put(47.5,9.5){$\bullet$}
\put(47.5,-3){$\bullet$}
\put(47.5,-15.5){$\bullet$}
\put(47.5,-28){$\bullet$}
\put(47.5,-40.5){$\bullet$}

\put(-40,34.5){$\bullet$}
\put(-27.5,22){$\bullet$}
\put(-15,9.5){$\bullet$}
\put(-2.5,-3){$\bullet$}

\put(35,-40.5){$\bullet$}
\put(22.5,-28){$\bullet$}
\put(10,-15.5){$\bullet$}

\put(35,34.5){$\bullet$}
\put(22.5,22){$\bullet$}
\put(10,9.5){$\bullet$}
\put(-15,-15.5){$\bullet$}
\put(-27.5,-28){$\bullet$}
\put(-40,-40.5){$\bullet$}

\end{picture}
\end{center}
\vspace{2.5cm}

\noindent Next, we seek an attractive basis  for $V$. Motivated by Lemma \ref{lem:nrel2}, we seek:
\begin{enumerate}
\item[\rm (i)] a basis of common eigenvectors for $\lbrace W_{-k} \rbrace_{k\in \mathbb N}$;
\item[\rm (ii)] a basis of common eigenvectors for $\lbrace W_{k+1} \rbrace_{k\in \mathbb N}$;
\item[\rm (iii)] a basis of common eigenvectors for $\lbrace G_{k+1} \rbrace_{k\in \mathbb N}$;
\item[\rm (iv)] a basis of common eigenvectors for $\lbrace {\tilde G}_{k+1} \rbrace_{k\in \mathbb N}$.
\end{enumerate}
Unfortunately, the above bases might not exist.  In Sections 10, 11 we will consider some TD systems of $q$-Serre type for which
the above bases do exist. The following result is used in Section 11.

\begin{lemma} \label{lem:assume} Assume that $V$ has a basis of type {\rm (i)} or {\rm (ii)}.
Then $V$ has a basis of type {\rm (iii)} and a basis of type {\rm (iv)}.
\end{lemma}
\begin{proof} Replacing $\Phi$ by $((\Phi^*)^\downarrow)^\Downarrow$ if necessary, we may assume that $V$
has a basis of type (i).
Let $0 \not=v \in V$ denote a common eigenvector for 
$\lbrace {W}_{-k}\rbrace_{k \in \mathbb N}$. Note that $v$ is an eigenvector for $W_0=A$, so $v \in E_jV$ for some $j$ $(0 \leq j \leq d)$.
For $k \in \mathbb N$ let $\omega_{-k}\in \mathbb F$ denote the $W_{-k}$ eigenvalue for $v$. Thus $\omega_0 = \theta_j$. 
Recall the decomposition $\lbrack 0^*, 0\rbrack$ of $V$, from Example \ref{thm:sixdecp} and Note \ref{note:u}. Denote this decomposition by
$\lbrace U_i \rbrace_{i=0}^d$.
 By Lemma \ref{lem:split}, $E_0V+\cdots + E_jV= U_{d-j}+\cdots + U_{d}$ and $E_0V+\cdots + E_{j-1}V= U_{d-j+1}+\cdots + U_{d}$.
 Define the vector $u=\Psi v$, where $\Psi$ is from Lemma \ref{lem:del}.
 By construction $u\in U_{d-j}$ and $u - v \in U_{d-j+1} + \cdots + U_d$.
We claim that for $k \in \mathbb N$, $u$ is an eigenvector for $G_{k}$ with eigenvalue $\omega_{-k} / \omega_0$. To prove the claim, let $k$ be given.
We show that $(G_k-\omega^{-1}_0 \omega_{-k} I) u = 0$. This is immediate if $k=0$, so assume that $k\geq 1$.
We have $0 = U_{d-j} \cap (U_{d-j+1} + \cdots + U_{d})$, so it suffices to show that
$(G_k-\omega^{-1}_0 \omega_{-k} I) u \in U_{d-j}$ and $(G_k-\omega^{-1}_0 \omega_{-k} I) u \in U_{d-j+1} + \cdots + U_{d}$. We have  $u \in U_{d-j}$, and
$G_k U_{d-j} \subseteq U_{d-j}$
by Theorem \ref{thm:GGaction}, so $(G_k-\omega^{-1}_0 \omega_{-k} I) u \in U_{d-j}$. For notational convenience, define $L_k = \lbrack W_0, G_k\rbrack/(b-1)$.
By the relation $\lbrack W_0, G_k \rbrack_b = (b-1)W_{-k}$ from Lemma \ref{lem:nrel1}, 
\begin{align*}
L_k = W_{-k} -W_0 G_k.
\end{align*}
By Proposition \ref{lem:fourP}(iii) we have $\lbrack W_0, L_k\rbrack_b =0$. Consequently $L_k E_jV \subseteq E_{j-1}V$, where $E_{-1}=0$.
 Note that 
 \begin{align*}
 &(G_k-\omega^{-1}_0 \omega_{-k}I) v=W_0^{-1} (W_0 G_k- W_{-k}) v=
 - W^{-1}_0 L_k v\in W^{-1}_0 L_k E_jV
 \\&\subseteq W^{-1}_0 E_{j-1}V = E_{j-1}V\subseteq  E_0V+\cdots + E_{j-1}V = U_{d-j+1} + \cdots + U_d.
 \end{align*}
We may now argue
\begin{align*}
(G_k-\omega^{-1}_0 \omega_{-k} I) u &=(G_k-\omega^{-1}_0 \omega_{-k}I) (u-v) + (G_k-\omega^{-1}_0 \omega_{-k}I) v
\\
&\in (G_k-\omega^{-1}_0 \omega_{-k} I) (U_{d-j+1}+\cdots + U_d) +  U_{d-j+1} + \cdots + U_{d}
\\
&= U_{d-j+1} + \cdots + U_{d}.
\end{align*}
\noindent We have shown that $(G_k - \omega^{-1}_0 \omega_{-k} I)u=0$, and the claim is proved.
By assumption, $V$ has a basis of type (i). We apply $\Psi$ to this basis, and get a new basis for $V$.
By the claim, the new basis has type (iii). We have shown that $V$ has a basis of type (iii).
By a similar argument, $V$ has a basis of type (iv).
\end{proof}

\section{Leonard systems of $q$-Serre type}
\noindent We continue to discuss the nontrivial TD system  $\Phi=(A;\lbrace E_i\rbrace_{i=0}^d;A^*;\lbrace E^*_i\rbrace_{i=0}^d)$ on $V$ of $q$-Serre type. Recall the $U^+_q$-module $V$ from
Corollary \ref{cor:modU}.
 Throughout this section we assume that
$\Phi$ is a Leonard system, so that
 $E_iV$ and $E^*_iV$ have dimension one for $0 \leq i \leq d$. For aesthetic reasons, we further assume that $q \in \mathbb F$ and 
\begin{align*}
 \theta_i = q^{2i-d}, \qquad \qquad \theta^*_i = q^{d-2i} \qquad \qquad (0 \leq i \leq d).
 \end{align*}
 \noindent Under the above assumptions, we describe how the alternating elements of $U^+_q$ act on $V$. 
\medskip

\noindent In \cite{qtet, evalTetq} the TD system $\Phi$ is described from the point of view of the $q$-tetrahedron algebra. We will use some results from these descriptions.
\begin{definition}\label{def:maps}
\rm Define $x_{12}=A$ and $x_{30}=A^*$. Define $x_{01} \in {\rm End}(V)$ such that for $0 \leq i \leq d$, the $i^{\rm th}$ 
component of $\lbrack D^*,D\rbrack$ is an eigenspace of $x_{01} $ with eigenvalue $q^{d-2i}$. Define $x_{23} \in {\rm End}(V)$ such that for $0 \leq i \leq d$, the $i^{\rm th}$ 
component of $\lbrack 0^*,0\rbrack$ is an eigenspace of $x_{23} $ with eigenvalue $q^{2i-d}$.
\end{definition}

\begin{lemma}\label{lem:qw} {\rm (See \cite[Theorem~10.4]{qtet}, \cite[Definition~4.1]{evalTetq}.)} We have
\begin{align*}
       & \frac{q x_{01} x_{12}- q^{-1} x_{12} x_{01}}{q-q^{-1}} = I, \qquad \qquad
           \frac{q x_{12} x_{23}- q^{-1} x_{23} x_{12}}{q-q^{-1}} = I, 
           \\
       & \frac{q x_{23} x_{30}- q^{-1} x_{30} x_{23}}{q-q^{-1}} = I, \qquad \qquad
           \frac{q x_{30} x_{01}- q^{-1} x_{01} x_{23}}{q-q^{-1}} = I.
\end{align*}
\end{lemma}

\begin{lemma} \label{lem:eval}{\rm (See   \cite[Lemma~9.4, Proposition~9.6]{evalTetq}.)} There exists a unique nonzero $\xi \in \mathbb F$ such that 
\begin{align*}
\xi (x_{01}-x_{23}) = \frac{\lbrack x_{30}, x_{12}\rbrack}{q-q^{-1}},
\qquad \qquad
\xi^{-1} (x_{12}-x_{30}) = \frac{\lbrack x_{01}, x_{23}\rbrack}{q-q^{-1}}.
\end{align*}
Moreover, $\xi$ is not among $q^{d-1}, q^{d-3},\ldots, q^{1-d}$.
\end{lemma}

\noindent Following \cite[Definition~9.10]{evalTetq}, define
\begin{align*}
\Upsilon = (q^{d+1}+q^{-d-1})I.
\end{align*}

\begin{lemma} 
\label{lem:upsfour} {\rm (See \cite[Lemma~9.11]{evalTetq}.)}
We have
\begin{align*}
&
\Upsilon = \xi(x_{01}x_{23}-I)+qx_{30}+q^{-1}x_{12},
 \qquad 
\Upsilon = \xi^{-1}(x_{12}x_{30}-I)+qx_{01}+q^{-1}x_{23}, 
\\
&
\Upsilon = \xi (x_{23}x_{01}-I)+qx_{12}+q^{-1}x_{30},
 \qquad 
\Upsilon = \xi^{-1}(x_{30}x_{12}-I)+qx_{23}+q^{-1}x_{01}.
\end{align*}
\end{lemma}

\noindent Next we use $x_{01}, x_{12}, x_{23}, x_{30}$ to describe how the alternating elements act on $V$.

\begin{lemma}\label{lem:si} There exists a sequence $\lbrace r_n \rbrace_{n \in \mathbb N}$ of scalars in $\mathbb F$ such that $r_0 = 1$ and for $n \in \mathbb N$ the following hold on the
$U^+_q$-module $V$:
\begin{align*}
W_{-n} &= r_n x_{12} - q \xi r_{n-1} I,\qquad \qquad 
G_{n} = r_n I - q \xi r_{n-1} x_{23},\\
W_{n+1} &= r_n x_{30} - q \xi r_{n-1}I, \qquad \qquad 
{\tilde G}_n  = r_n I - q \xi r_{n-1} x_{01}.
\end{align*}
In the above lines $r_{-1}=0$.
\end{lemma}
\begin{proof} We use induction with respect to the ordering of the alternating elements given in Lemma \ref{lem:nrecgen}. We carry out the induction
using \eqref{eq:nsolvG}--\eqref{eq:nsolvWp} and Lemmas \ref{lem:qw}--\ref{lem:upsfour}.
The details are given below.
\\
\noindent $G_n$ $(n\geq 1)$: We evaluate the right-hand side of  \eqref{eq:nsolvG}. By induction,
\begin{align*}
\sum_{k=0}^{n-1} W_{-k} W_{n-k} q^{n-1-2k} &= 
\sum_{k=0}^{n-1}(r_k x_{12} - q \xi r_{k-1} I)( r_{n-k-1} x_{30} - q \xi r_{n-k-2}I)q^{n-1-2k}
\\
&= 
x_{12}x_{30} \sum_{k=0}^{n-1}r_k  r_{n-k-1} q^{n-1-2k}
-q \xi x_{12}\sum_{k=0}^{n-1}r_k r_{n-k-2}q^{n-1-2k}
\\
& \qquad 
-q \xi x_{30}\sum_{k=0}^{n-1}r_{k-1} r_{n-k-1} q^{n-1-2k}
+q^2 \xi^2\sum_{k=0}^{n-1}r_{k-1}r_{n-k-2}q^{n-1-2k}.
\end{align*}
By Lemma  \ref{lem:upsfour},
\begin{align*}
x_{12} x_{30}  = (1 + \xi q^{d+1} + \xi q^{-d-1}) I - q \xi x_{01} - q^{-1} \xi x_{23}.
\end{align*}
\noindent By induction,
\begin{align*}
\sum_{k=1}^{n-1} G_k  \tilde G_{n-k} q^{n-2k} &=   
\sum_{k=1}^{n-1} (r_k I - q \xi r_{k-1} x_{23})( r_{n-k} I - q \xi r_{n-k-1} x_{01})q^{n-2k}
\\
&= 
q^2 \xi^2 x_{23} x_{01} \sum_{k=1}^{n-1} r_{k-1} r_{n-k-1}q^{n-2k}
-q \xi x_{23} \sum_{k=1}^{n-1}  r_{k-1} r_{n-k} q^{n-2k}
\\
&
\qquad -q \xi x_{01}\sum_{k=1}^{n-1} r_k r_{n-k-1} q^{n-2k}
+
\sum_{k=1}^{n-1} r_k r_{n-k} q^{n-2k}.
\end{align*}
By Lemma  \ref{lem:upsfour},
\begin{align*}
x_{23} x_{01}  = (1+ \xi^{-1} q^{d+1} + \xi^{-1} q^{-d-1}) I -q\xi^{-1} x_{12} - q^{-1} \xi^{-1}x_{30}.
\end{align*}
By induction and Lemma \ref{lem:eval},
\begin{align*}
W_n  W_0-W_0 W_n &= \lbrack r_{n-1} x_{30} - q \xi r_{n-2}I, x_{12} \rbrack
\\
&=  r_{n-1} \lbrack x_{30}, x_{12} \rbrack
\\&= (q-q^{-1})\xi r_{n-1} (x_{01}- x_{23}).
\end{align*}
\noindent Evaluating the right-hand side of  \eqref{eq:nsolvG} using the above comments, we express $G_n$ is a linear combination of $x_{01}$, $x_{12}$, $x_{23}$, $x_{30}$, $I$. In this linear combination
the coefficients of $x_{01}$, $x_{12}$, $x_{23}$, $x_{30}$ are $0$, $0$, $-q \xi r_{n-1}$, $0$, respectively. Therefore
$G_n + q \xi r_{n-1} x_{23}$ is a scalar multiple of $I$. Denoting this scalar by $r_n$, we have
$G_{n} = r_n I - q \xi r_{n-1} x_{23}$.
\\
\noindent ${\tilde G}_n$ $(n \geq 1)$: We evaluate  \eqref{eq:nsolvGt}. By induction and Lemma \ref{lem:eval},
\begin{align*}
{\tilde G}_n &= G_n + \frac{\lbrack W_0, W_n \rbrack}{1- b^{-1}}\\
&= r_n I - q \xi r_{n-1} x_{23} + \frac{ \lbrack x_{12}, r_{n-1} x_{30} - q \xi r_{n-2} I \rbrack }{1-b^{-1}}
\\
&= r_n I - q \xi r_{n-1} x_{23} + r_{n-1} \frac{\lbrack x_{12}, x_{30} \rbrack }{1-b^{-1}}
\\
&= r_n I - q \xi r_{n-1} x_{23} +q \xi  r_{n-1} (x_{23}-x_{01})
\\
&= r_n I - q \xi r_{n-1} x_{01}.
\end{align*}
\noindent $W_{-n}$ $(n\geq 1)$: We evaluate \eqref{eq:nsolvWm}. By induction and Lemma \ref{lem:qw},
\begin{align*}
W_{-n} &= \frac{\lbrack W_0,G_n\rbrack_b}{b-1}
\\
&= \frac{\lbrack x_{12},  r_n I - q \xi r_{n-1} x_{23}\rbrack_b}{b-1} \\
&=  r_n x_{12} - q \xi  r_{n-1} \frac{\lbrack x_{12}, x_{23}\rbrack_b}{b-1} \\
&=  r_n x_{12} - q \xi  r_{n-1}I.
\end{align*}
\\
\noindent $W_{n+1}$ $(n \geq 1)$: We evaluate \eqref{eq:nsolvWp}. By induction and Lemma \ref{lem:qw},
\begin{align*}
W_{n+1} &= \frac{\lbrack G_n, W_1\rbrack_b}{b-1}
\\
&= \frac{\lbrack  r_n I - q \xi r_{n-1} x_{23}, x_{30} \rbrack_b}{b-1} \\
&=  r_n x_{30} - q \xi  r_{n-1} \frac{\lbrack x_{23}, x_{30}\rbrack_b}{b-1} \\
&=  r_n x_{30} - q \xi  r_{n-1}I.
\end{align*}
 \end{proof}

\noindent Our next main goal is to compute the scalars $\lbrace r_n \rbrace_{n \in \mathbb N}$ from Lemma \ref{lem:si}.
To compute these scalars, it is convenient to make a change of variables.
\begin{definition}\label{def:rv} \rm For $n \in \mathbb N$ define
\begin{align}
r^\vee_n = \sum_{k=0}^n r_k r_{n-k} q^{n-2k}. \label{eq:rv}
\end{align}
Note that $r^\vee_0=1$.
\end{definition}

\noindent Next we compute $\lbrace r^\vee_n \rbrace_{n \in \mathbb N}$. This is done in the following recursive manner.

\begin{proposition} \label{lem:rvRec} For $n\geq 1$ we have
\begin{align}
\label{eq:recrv}
0 &= r^\vee_n - q(1+ \xi q^{d+1} + \xi q^{-d-1}) r^\vee_{n-1} + q^2 \xi (\xi + q^{d+1} + q^{-d-1}) r^\vee_{n-2} - q^3 \xi^2  r^\vee_{n-3},
\end{align}
\noindent where $r^\vee_0=1$ and $r^\vee_{-1}=0$ and $r^\vee_{-2}=0$.
\end{proposition}
\begin{proof}Let $\zeta_n$ denote the expression on the right in \eqref{eq:recrv}. We show that $\zeta_n=0$.
Evaluate \eqref{eq:nGGWW1} using
Lemma \ref{lem:si}, and simplify the result using Definition \ref{def:rv}. This yields
\begin{align*}
0 &= r^\vee_n I - q r^\vee_{n-1}(x_{12} x_{30} + q \xi x_{01} + q^{-1} \xi x_{23})  \\
& \qquad \qquad \quad + q^2 \xi r^\vee_{n-2} (\xi x_{23} x_{01} + qx_{12} + q^{-1} x_{30} ) - q^3 \xi^2 r^\vee_{n-3} I.
\end{align*}
By Lemma  \ref{lem:upsfour},
\begin{align*}
x_{12} x_{30} + q \xi x_{01} + q^{-1} \xi x_{23} &= (1 + \xi q^{d+1} + \xi q^{-d-1}) I,\\
\xi x_{23} x_{01} + qx_{12} + q^{-1} x_{30} &= (\xi + q^{d+1} + q^{-d-1}) I.
\end{align*}
By these comments  $\zeta_nI=0$, so $\zeta_n=0$.
\end{proof}

\begin{example}\rm We have
\begin{align*}
 r^\vee_0 &=1,
 \\
 r^\vee_1 &= q+ \xi q^{d+2}+ \xi q^{-d},
 \\
 r^\vee_2 &=  q^2+ \xi^2 q^{2d+4} + \xi^2 q^{-2d} + \xi q^{d+3} +\xi q^{1-d} + \xi^2 q^2.
 \end{align*}
 \end{example}

\noindent Next we obtain the scalars $\lbrace r_n \rbrace_{n \in \mathbb N}$ from the scalars $\lbrace r^\vee_n \rbrace_{n \in \mathbb N}$.
This is done in the following recursive manner.

\begin{proposition} \label{prop:recover} For $n \geq 1$,
\begin{align*}
 r_n = \frac{ r^\vee_n - \sum_{k=1}^{n-1} r_k r_{n-k}q^{n-2k}}{q^n+q^{-n}}.
 \end{align*}
 \end{proposition}
 \begin{proof} This is a reformulation of \eqref{eq:rv}.
 \end{proof}

\begin{example}\rm We have
\begin{align*}
 r_0 &=1,
 \\
 r_1 &= \frac{q+ \xi q^{d+2}+ \xi q^{-d}}{q+q^{-1}},
 \\
 r_2 &= \frac{q^2+ \xi^2 q^{2d+4}+\xi^2 q^{-2d}}{(q+q^{-1})^2 (q^2+q^{-2})} + 
 \frac{q^2+ \xi^2 q^{2d+4} + \xi^2 q^{-2d} + \xi q^{d+3} +\xi q^{1-d} + \xi^2 q^2}{(q+q^{-1})^2}.
 \end{align*}
 \end{example}

\noindent Next we describe $\lbrace r_n \rbrace_{n \in \mathbb N}$ and $\lbrace r^\vee_n \rbrace_{n \in \mathbb N}$ using generating functions.

\begin{definition}\label{def:gf} \rm Define the generating functions
\begin{align*}
 R(t) = \sum_{n \in \mathbb N} r_n t^n, \qquad \qquad R^\vee(t) = \sum_{n \in \mathbb N} r^\vee_n t^n.
 \end{align*}
\end{definition}

\begin{lemma} \label{lem:SR2} We have
\begin{align*}
R^\vee(t) = R(qt)R(q^{-1}t).
\end{align*}
\end{lemma}
\begin{proof} This is a routine consequence of \eqref{eq:rv}.
\end{proof}

\noindent The recursion \eqref{eq:recrv}
 looks as follows in terms of generating functions.
\begin{proposition} \label{lem:RVform} We have
\begin{align*}
R^\vee(t) = \frac{1}{(1-qt)(1- q^{d+2}\xi t)(1-  q^{-d}\xi t)}.
\end{align*}
\end{proposition}
\begin{proof} Consider the generating function  $G(t)=\sum_{n\in \mathbb N} c_n t^n$  such that $c_0= r^\vee_0$ and
for $n\geq 1$, $c_n$ is equal to the expression on the right in  \eqref{eq:recrv}. Using Proposition  \ref{lem:rvRec}
we obtain
\begin{align*}
1 &=G(t)
\\
&= R^\vee(t) - q(1 + \xi q^{d+1}+\xi q^{-d-1}) t R^\vee(t) + q^2 \xi (\xi + q^{d+1} + q^{-d-1}) t^2 R^\vee(t) - q^3 \xi^2 t^3 R^\vee(t)
\\
&= (1-qt)(1-  q^{d+2}\xi t)(1-  q^{-d}\xi t) R^\vee(t).
\end{align*}
The result follows.
\end{proof}

\noindent We now give our third main result.
\begin{theorem} \label{thm:third} We have
\begin{align*}
R(qt)R(q^{-1}t ) = \frac{1}{(1-qt)(1-q^{d+2}\xi t)(1- q^{-d}\xi t)}.
\end{align*}
\end{theorem} \begin{proof}
By Lemma  \ref{lem:SR2} and Proposition \ref{lem:RVform}.
\end{proof}

\noindent Next we impose a mild assumption on $\mathbb F$, and give an explicit formula for $R(t)$.
The following is our fourth main result.
\begin{theorem}\label{thm:exp}
Assume that $\mathbb F$ has characteristic 0. Then
 \begin{align*}
 R(t) =
{\rm exp}  \Biggl( \sum_{k=1}^\infty  \frac{1+ q^{k(d+1)}\xi^k + q^{-k(d+1)}\xi^k}{q^k+q^{-k}} \,\frac{q^k t^k}{k} \Biggr).
\end{align*}
\end{theorem}
\begin{proof} Define the generating function
 \begin{align*}
 P(t) = \sum_{k=1}^\infty  \frac{1+ q^{k(d+1)}\xi^k +q^{-k(d+1)}\xi^k}{q^k+q^{-k}} \,\frac{q^k t^k}{k}.
\end{align*}
Note that $P(t)$ has constant term 0.
Define $E(t) = {\rm exp} \,P(t)$, and note that $E(t)$ has constant term 1. We will show that $R(t)=E(t)$. By  Lemma
 \ref{lem:SR2} and the construction,
 $R(t)$ is the unique generating function over $\mathbb F$ that has constant term 1 and $R^\vee(t)=R(qt)R(q^{-1}t)$. Therefore,  $R^\vee(t)=E(qt)E(q^{-1}t)$ implies $R(t)=E(t)$.
Recall the natural logarithm ln; see for example \cite[Section~2]{beckPBW}. We have
\begin{align*}
E(qt) E(q^{-1}t) &= {\rm exp}\,\Bigl( P(qt)+ P(q^{-1}t)\Bigr)
\\
&=  
{\rm exp} \Biggl(  \sum_{k=1}^\infty  \frac{1+ q^{k(d+1)} \xi^k+ q^{-k(d+1)}\xi^k }{k} \,q^k t^k \Biggr)
 \\
 &={\rm exp} \Biggl(  \sum_{k=1}^\infty  \frac{q^k t^k}{k} \Biggr)  
 {\rm exp} \Biggl(  \sum_{k=1}^\infty  \frac{q^{k(d+2)} \xi^k t^k}{k} \Biggr)
  {\rm exp} \Biggl(  \sum_{k=1}^\infty  \frac{q^{-kd} \xi^k t^k}{k} \Biggr) 
  \\
& ={\rm exp} \Bigl( - {\rm ln} (1-qt)\Bigr)
{\rm exp}\Bigl(  - {\rm ln} (1-q^{d+2} \xi t)\Bigr) 
{\rm exp} \Bigl( - {\rm ln} (1-q^{-d} \xi t) \Bigr)
 \\
 &=
 (1-qt)^{-1}(1- q^{d+2}\xi t)^{-1}(1-  q^{-d}\xi t)^{-1} \\
 &= R^\vee(t),
\end{align*}
and consequently $R(t)=E(t)$.

\end{proof}
\noindent For the sake of completeness, we mention some matrix representations of the maps in Definition \ref{def:maps}.
\begin{remark}\rm In \cite[Section~10]{evalTetq} we defined 24 bases for $V$, and for each basis we gave the matrices that represent  $x_{01}$, $x_{12}$, $x_{23}$, $x_{30}$. For one of these bases the representing matrices are
shown below.
\bigskip

\centerline{
\begin{tabular}[t]{c|cccc}
   map & $x_{01}$ & $x_{12}$ & $x_{23}$ & $x_{30}$ 
   \\
   \hline
   matrix &
     $ S_{q^{-1}}(\xi^{-1})$ &$Z E_q Z$ & $K_{q^{-1}}$ & $G_{q^{-1}} (\xi)$
									             \end{tabular}}
\medskip

\noindent The definitions of $Z$, $K_q$,  $E_q$, $G_q(\xi)$, $S_q(\xi)$ can be found in \cite[Appendix A]{evalTetq}.							           
\end{remark}

\section{Distance-regular graphs}

\noindent Recall the field $\mathbb R$ of real numbers.
Throughout this section, assume that $\mathbb F = \mathbb R$. 
\medskip

\noindent 
In the topic of algebraic graph theory, there is a family of finite  undirected graphs, said to be distance-regular \cite{banIto}, \cite{bbit}, \cite{BCN}, \cite{dkt}, \cite{drg}. We refer the reader to these works, 
for background information about the concepts and notation used below.
\medskip

\noindent
There is a kind of distance-regular
graph, said to have classical parameters $(d, b, \alpha, \beta)$; see \cite[Section~6.1]{BCN}. The parameter $d$ is the diameter of the graph \cite[p.~433]{BCN}. 
The parameters $b,\alpha, \beta $ are real numbers used to describe the intersection numbers of the graph \cite[p.~1]{BCN}.
Throughout this section, we fix a distance-regular graph $\Gamma$  that has diameter $d\geq 3$ and classical parameters $(d,b,\alpha, \beta)$ 
with $b \not=1$ and $\alpha=b-1$.
The condition on $\alpha$ implies that $\Gamma$ is formally self-dual in the sense of \cite[p.~49]{BCN}. By \cite[Proposition~6.2.1]{BCN} $b$ is an integer and $b \not=0$, $b \not=-1$. Note that $b$ is not a root of unity.
Let $X$ denote the
vertex set of $\Gamma$. Let ${\rm Mat}_X(\mathbb R)$ denote the algebra  of  matrices that have rows and columns indexed by $X$ and all entries in $\mathbb R$.
Let $\mathbb V=\mathbb R^X$
denote the vector space  consisting of the column vectors whose coordinates are indexed
by $X$  and whose entries are in $\mathbb R$. Note that ${\rm Mat}_X(\mathbb R)$ acts on $\mathbb V$ by left multiplication.
 We endow $\mathbb V$ with the bilinear  form $\langle\,,\,\rangle$ that satisfies $\langle u,v\rangle = u^t v$ for $u,v \in \mathbb V$, where $t$ denotes transpose.
 Note that $\langle\,,\,\rangle$ is symmetric.
   For $B \in {\rm Mat}_X(\mathbb R)$,
 \begin{align}
 \label{ex:uBu}
  \langle Bu,v\rangle = \langle u, B^t v\rangle \qquad \qquad u,v\in \mathbb V.
  \end{align}
Let $\mathbb A \in {\rm Mat}_X(\mathbb R)$ denote the adjacency matrix of $\Gamma$ \cite[Section~7]{drg}. The matrix $\mathbb A$ is symmetric, and each entry is $0$ or $1$.
For the rest of this section, fix $x \in X$  and let $\mathbb A^*=\mathbb A^*(x) \in {\rm Mat}_X(\mathbb R)$ denote the dual adjacency matrix of $\Gamma$ with respect to $x$ \cite[Section~7]{drg}. The matrix $\mathbb A^*$ 
is diagonal. 
Let $\mathbb T = \mathbb T(x)$ denote the subalgebra of ${\rm Mat}_X(\mathbb R)$ generated by $\mathbb A, \mathbb A^*$. The algebra $\mathbb T$ is called
the {\it subconstituent algebra} (or {\it Terwilliger algebra}) of $\Gamma$ with respect to $x$; see \cite[Definition~3.3]{terwSub1}. By construction, $\mathbb T$ is closed
under the transpose map. 
\medskip


\noindent We comment on the $\mathbb T$-modules.
By a {\it $\mathbb T$-module}, we mean a subspace $V \subseteq \mathbb V$ such that $\mathbb T V \subseteq V$. Let $V$ denote a $\mathbb T$-module, and let $V'$ denote a $\mathbb T$-module contained in
$V$. Then 
by \cite[p.~802]{goTer}, the orthogonal complement of $V'$ in $V$ is a $\mathbb T$-module. Consequently, each $\mathbb T$-module is an orthogonal direct sum
of irreducible $\mathbb T$-modules.  In particular, the $\mathbb T$-module $\mathbb V$ is an orthogonal direct sum of irreducible $\mathbb T$-modules. 
 By \cite[Example~1.4]{TD00} the elements $\mathbb A$, $\mathbb A^*$
act on each irreducible $\mathbb T$-module as a TD pair.
\medskip

\noindent For convenience,  we now adjust $\mathbb A$ and $ \mathbb A^*$.
By \cite[Corollary~8.4.4]{BCN}, for  $\mathbb A$ and $\mathbb A^*$ the roots of the minimal polynomial have the form
\begin{align*}
r b^{-i} + s \qquad \qquad (0 \leq i \leq d), 
\end{align*}
where $r,s \in \mathbb R$ and $r\not=0$.
Define $A, A^* \in {\rm Mat}_X(\mathbb R)$ such that
\begin{align*}
\mathbb A = A+ s I, \qquad \qquad \mathbb A^* = A^* + sI.
\end{align*}
By construction, for $A$ and $A^*$ the roots of the minimal polynomial are $\lbrace r b^{-i}\rbrace_{i=0}^d$.
By construction, $A$ and $A^*$ are symmetric. 
By construction, the algebra $\mathbb T$ is generated by $A, A^*$.
By \cite[Lemma~9.4]{drg}, both
\begin{align*}
A^3 A^* - (b+ b^{-1}+1) A^2 A^*A+ (b+ b^{-1}+1) A A^* A^2 - A^* A^3 &= 0,\\
A^{* 3} A - (b+ b^{-1}+1) A^{* 2}  AA^*+ (b+ b^{-1}+1) A^* A A^{*2} - A A^{*3} &= 0.
\end{align*}
Thus
$A$, $A^*$ satisfy the 
$q$-Serre relations, where $q$ is a complex number such that $q^2=b$. The scalar $q$ is nonzero, and not a root of unity. We caution the reader that $q \not\in \mathbb R$ if $b<-1$.

\begin{lemma}\label{lem:aasg} There exists an algebra homomorphism $\natural: U^+_q \to \mathbb T$ that sends $W_0 \mapsto A$ and 
$W_1 \mapsto A^*$. The map $\natural $ is surjective.
\end{lemma}
\begin{proof} 
The matrices $A, A^*$  satisfy the $q$-Serre relations. Moreover $A, A^*$ generate $\mathbb T$.
\end{proof}

\noindent
In Lemma \ref{lem:nAAut} we mentioned an
antiautomorphism $\dagger$ of $U^+_q$ that fixes each of $W_0$, $W_1$. 
\begin{lemma} \label{lem:cd}
The following diagram commutes:
\begin{equation*}
{\begin{CD}
U^+_q @>\natural  >>
               {\mathbb T}
              \\
         @V \dagger VV                   @VV t V \\
         U^+_q @>>\natural >
                                  {\mathbb  T}
                        \end{CD}} \qquad \qquad {\mbox{\rm $t$ = transpose.}}
\end{equation*}
\end{lemma}
\begin{proof} Chase the generators $W_0$, $W_1$ of $U^+_q$ around the diagram, using the fact that $A^t=A$ and $(A^*)^t=A^*$.
\end{proof}
\begin{definition} \label{def:agt}
\rm Recall the alternating elements of $U^+_q$ from Section 8. For each alternating element, we retain the same notation for its image under $\natural$.
These images will be called the {\it alternating elements of $\mathbb T$}.
\end{definition}

\noindent We just defined the alternating elements
\begin{align*}
\lbrace  W_{-k}\rbrace_{k \in \mathbb N}, \quad 
\lbrace  W_{k+1}\rbrace_{k \in \mathbb N}, \quad
\lbrace  G_{k+1}\rbrace_{k\in \mathbb N}, \quad
\lbrace {\tilde G}_{k+1}\rbrace_{k \in \mathbb N}
\end{align*}
of $\mathbb T$. By construction, these elements have all entries in $\mathbb R$.
In the next two lemmas we emphasize some additional features of these elements.
\begin{lemma}
For $k, \ell \in \mathbb N$ the
following relations hold in $\mathbb T$:
\begin{align*}
&
\lbrack  W_{-k},  W_{-\ell}\rbrack=0,  \qquad 
\lbrack  W_{k+1},  W_{\ell+1}\rbrack= 0,
\\
&
\lbrack  G_{k+1},  G_{\ell+1}\rbrack=0,
\qquad 
\lbrack {\tilde G}_{k+1},  {\tilde G}_{\ell+1}\rbrack= 0.
\end{align*}
\end{lemma}
\begin{proof} By Lemma \ref{lem:nrel2}.
\end{proof}

\begin{lemma} \label{lem:tran}  Referring to the alternating elements of $\mathbb T$,
the following hold for $k \in \mathbb N$:
\begin{enumerate}
\item[\rm (i)] $W_{-k}$ and $W_{k+1}$ are symmetric;
\item[\rm (ii)] $G_{k+1}$ and ${\tilde G}_{k+1}$ are the transposes of each other.
\end{enumerate}
\end{lemma}
\begin{proof} By Lemma \ref{lem:nsigSact}(ii) and Lemma \ref{lem:cd}.
\end{proof}

\noindent We have seen that the alternating elements $\lbrace W_{-k}\rbrace_{k \in \mathbb N}$  of $\mathbb T$ are mutually commuting,
symmetric, and have all entries in $\mathbb R$. Therefore the alternating elements $\lbrace W_{-k}\rbrace_{k \in \mathbb N}$  of $\mathbb T$ can be simultaneously diagonalized. Similar comments
apply to the alternating elements $\lbrace W_{k+1}\rbrace_{k \in \mathbb N}$  of $\mathbb T$.
\medskip

\noindent We now give our fifth main result.

\begin{theorem} \label{thm:orthog} Each  irreducible $\mathbb T$-module is an orthogonal direct sum of its common eigenspaces for $\lbrace W_{-k}\rbrace_{k \in \mathbb N}$,
and an orthogonal direct sum of its common eigenspaces for $\lbrace W_{k+1}\rbrace_{k \in \mathbb N}$.
\end{theorem}
\begin{proof} By the comments above the theorem statement, and since the eigenspaces of a real symmetric matrix are mutually orthogonal.
\end{proof}

\noindent We now give our sixth main result.
\begin{theorem} \label{thm:orthog2} Let $V$ denote an  irreducible $\mathbb T$-module. 
Then $V$ is a direct sum of its common eigenspaces for $\lbrace G_{k+1}\rbrace_{k \in \mathbb N}$,
and a direct sum of its common eigenspaces for $\lbrace {\tilde G}_{k+1}\rbrace_{k \in \mathbb N}$.
 For both of these  direct sums, each summand  is
nonorthogonal to a unique summand in the other direct sum. Let $W$ and $\tilde W$ denote nonorthogonal summands in the first direct sum and second direct sum, respectively.
Then for $k \in \mathbb N$ the eigenvalue of $G_{k+1}$ for $W$ is equal to the eigenvalue of $\tilde G_{k+1}$ for $\tilde W$.
\end{theorem}
\begin{proof} The two direct sums exist by Lemma \ref{lem:assume} and Theorem \ref{thm:orthog}.
The other assertions follow from  \eqref{ex:uBu} and Lemma  \ref{lem:tran}(ii).
\end{proof}

\begin{note}\rm (See
\cite[Example~8.4]{drg}.)
The following distance-regular graphs have classical parameters $(d,b,\alpha, \beta)$ with $b \not=1$ and $\alpha=b-1$:
 \begin{enumerate}
 \item[\rm (i)]
the bilinear forms graph \cite[p.~280]{BCN};
 \item[\rm (ii)] the alternating forms graph \cite[p.~282]{BCN};
 \item[\rm (iii)] the Hermitean forms graph \cite[p.~285]{BCN};
 \item[\rm (iv)] the quadratic forms graph \cite[p.~290]{BCN},
 \item[\rm (v)] the affine $E_6$ graph \cite[p.~340]{BCN};
\item[\rm (vi)] the extended ternary Golay code graph \cite[p.~359]{BCN}. 
\end{enumerate}
\end{note}
\medskip

\section{Directions for further research}

\noindent In this section we list some conjectures and open problems.

\begin{problem}\rm Referring to the distance-regular graph $\Gamma$ and the  subconstituent algebra $\mathbb T$ from Section 11,
let $V$ denote an irreducible $\mathbb T$-module. Consider
 four bases for $V$, obtained by taking a common eigenbasis for $\lbrace W_{-k}\rbrace_{k \in \mathbb N}$,
$\lbrace W_{k+1}\rbrace_{k \in \mathbb N}$,
$\lbrace G_{k+1}\rbrace_{k \in \mathbb N}$, 
$\lbrace {\tilde G}_{k+1}\rbrace_{k \in \mathbb N}$, respectively. Normalize the four bases in an attractive fashion.
For all pairs of bases among the four, find the transition matrices.
For each of the four bases, find the matrices that represent the alternating elements of $\mathbb T$.
\end{problem}

\begin{conjecture} \rm Referring to the distance-regular graph $\Gamma$ and the  subconstituent algebra $\mathbb T$ from Section 11,
for every irreducible $\mathbb T$-module the common eigenspaces for
$\lbrace W_{-k}\rbrace_{k \in \mathbb N}$ or
$\lbrace W_{k+1}\rbrace_{k \in \mathbb N}$ or $\lbrace G_{k+1}\rbrace_{k \in \mathbb N}$
or $\lbrace {\tilde G}_{k+1}\rbrace_{k \in \mathbb N}$ all have dimension one.
\end{conjecture}

\begin{problem}\rm We refer to the nontrivial TD system $\Phi$ on $V$ that has $q$-Serre type, as in Section 9.
 In \cite{tdanduqsl2hat} we used $\Phi$ to turn $V$ into a $U\sb q(\widehat{\mathfrak{sl}}\sb 2)$-module.
 We did this in two ways; see \cite[Theorems~13.1,13.2]{tdanduqsl2hat}.
 For the  $U\sb q(\widehat{\mathfrak{sl}}\sb 2)$-module $V$ in \cite[Theorems~13.1]{tdanduqsl2hat},
 the weight space decomposition is 
 $\lbrack 0^*, D\rbrack$. For the $U\sb q(\widehat{\mathfrak{sl}}\sb 2)$-module $V$ in \cite[Theorem~13.2]{tdanduqsl2hat}, the weight space decomposition is 
 $\lbrack D^*, 0\rbrack$.
Assume that $\mathbb F$ is algebraically closed and
characteristic zero. By \cite{charp},  the $U\sb q(\widehat{\mathfrak{sl}}\sb 2)$-module $V$ is a tensor product of evaluation modules.
An evaluation module is a direct sum of its weight spaces, and these weight spaces all have dimension one. Thus the
$U\sb q(\widehat{\mathfrak{sl}}\sb 2)$-module $V$ becomes a direct sum,
with each summand  a tensor product of evaluation module weight spaces. These summands have dimension one.
Investigate how these
summands are related to the common eigenspaces for $\lbrace W_{-k}\rbrace_{k \in \mathbb N}$ or
$\lbrace W_{k+1}\rbrace_{k \in \mathbb N}$ or $\lbrace G_{k+1}\rbrace_{k \in \mathbb N}$
or $\lbrace {\tilde G}_{k+1}\rbrace_{k \in \mathbb N}$.
The article \cite[Section~4.2]{basAlt} might be useful in this direction.
\end{problem}

\begin{problem}\rm Our main results are about TD pairs of $q$-Serre type. Find analogous results for general TD pairs.
\end{problem}

\section{Acknowledgement} The author thanks Pascal Baseilhac for many discussions about $U^+_q$ and its alternating elements.
\section{Funding} This research did not receive any specific grant from funding agencies in the public, commercial, or not-for-profit sectors.
\section{Declarations of competing interest} No conflicts of interest.

\bigskip

\noindent Paul Terwilliger \hfil\break
\noindent Department of Mathematics \hfil\break
\noindent University of Wisconsin \hfil\break
\noindent 480 Lincoln Drive \hfil\break
\noindent Madison, WI 53706-1388 USA \hfil\break
\noindent email: {\tt terwilli@math.wisc.edu }\hfil\break

\end{document}